\documentclass[11pt]{amsart} 

\usepackage{amssymb,amsmath,url, stackrel}
\usepackage{hyperref,color}
\usepackage{amsfonts}
\usepackage{amscd}
\usepackage{graphicx}
\usepackage[all]{xy}
\usepackage[table]{xcolor}
\usepackage{enumerate}
\usepackage{tikz}

 \theoremstyle{plain}
\newtheorem{thm}{Theorem}[section]
  \theoremstyle{plain}
  \newtheorem{prop}[thm]{Proposition}
  \theoremstyle{plain}
  \newtheorem{cor}[thm]{Corollary}
  \theoremstyle{plain}
  \newtheorem{lem}[thm]{Lemma}
\theoremstyle{plain}
  \newtheorem{rem}[thm]{Remark}
\theoremstyle{plain}
  
\theoremstyle{plain}
  \newtheorem{defn}[thm]{Definition}
\theoremstyle{plain}
  
  \theoremstyle{plain}
  
    \theoremstyle{plain}
  \newtheorem{exam}[thm]{Example}
  
\renewcommand{\exp}{\mathrm{Exp}}

\newcommand{\bbR}{\mathbb{R}}
\def\co{\colon\thinspace}
\newcommand{\injects}{\hookrightarrow}
\newcommand{\homeo}{\cong}
\newcommand{\surjects}{\twoheadrightarrow}
\newcommand{\isom}{\cong}
\newcommand{\leqs}{\leqslant}
\newcommand{\geqs}{\geqslant}
\newcommand{\cross}{\times}
\newcommand{\heq}{\simeq}
\newcommand{\maps}{\longrightarrow}

\newcommand{\srm}[1]{\stackrel{#1}{\maps}}
\newcommand{\srt}[1]{\stackrel{#1}{\to}}

\newcommand{\e}{\emph}

\newcommand{\wt}[1]{\widetilde{#1}}

\newcommand{\stab}{\mathrm{Stab}}
\newcommand{\C}{\mathbb{C}}
\newcommand{\R}{\mathbb{R}}
\newcommand{\aq}{/\!\!/}
\newcommand{\X}{\mathfrak{X}}

\renewcommand{\hom}{\mathrm{Hom}}
\newcommand{\SL}{\mathrm{SL}}
\newcommand{\GL}{\mathrm{GL}}
\newcommand{\SO}{\mathrm{SO}}
\newcommand{\Sp}{\mathrm{Sp}}
\newcommand{\SU}{\mathrm{SU}}

\newcommand{\F}{\mathtt{F}}

\newcommand{\XC}[1]{\mathfrak{X}_{#1}}

\newcommand{\Z}{\mathbb{Z}}

\newcommand{\Ad}{\mathrm{Ad}}

\newcommand{\p}{\mathrm{P}} 
\newcommand{\codim}{\mathrm{codim}}

\AtBeginDocument{%
   \def\MR#1{}
}

\makeatother

\begin{document}

\title[Homotopy of Character Varieties]{Bad Representations and Homotopy of Character Varieties}
 
\author[C. Gu\'erin]{Cl\'ement Gu\'erin}

\address{Science and Technology Department, Mayotte University Center, Route nationale 3, BP 53, 97660 Dembeni, France}

\email{clement.guerin@univ-mayotte.fr}
 
\author[S. Lawton]{Sean Lawton}

\address{Department of Mathematical Sciences, George Mason University,
4400 University Drive,
Fairfax, Virginia  22030, USA}

\email{slawton3@gmu.edu}

\author[D. Ramras]{Daniel Ramras}

\address{Department of Mathematical Sciences, Indiana University-Purdue University Indianapolis, 402 N. Blackford, 
Indianapolis, IN 46202, USA}

\email{dramras@math.iupui.edu}

%\thanks{}



%14L24   	Geometric invariant theory [See also 13A50]
%14L30  	Group actions on varieties or schemes (quotients) [See also 13A50, 14L24, 14M17]
%14B05  	Singularities [See also 14E15, 14H20, 14J17, 32Sxx, 58Kxx]
%55Q05  	Homotopy groups, general; sets of homotopy classes
%14D20  	Algebraic moduli problems, moduli of vector bundles {For analytic moduli problems, see 32G13}
%55U10   	Simplicial sets and complexes

\begin{abstract}Let $G$ be a connected reductive complex affine algebraic group, and let  $\XC{r}$ denote the moduli space of $G$-valued representations of a rank $r$ free group.  We first characterize the singularities in $\XC{r}$, extending a theorem of Richardson and proving a Mumford-type result about topological singularities; this resolves conjectures of Florentino-Lawton.  In particular, we  compute the codimension of the orbifold singular locus using facts about Borel-de Siebenthal subgroups.  We then use the codimension bound to calculate higher homotopy groups of the smooth locus of $\XC{r}$, proving conjectures of Florentino-Lawton-Ramras.  Lastly, using the earlier analysis of Borel-de Siebenthal subgroups, we prove a conjecture of Sikora about centralizers of irreducible representations in Lie groups. 
\end{abstract}

% arXiv abstract:

% Let G be a connected reductive complex affine algebraic group, and let X denote the moduli space of G-valued representations of a rank r free group.  We first characterize the singularities in X, extending a theorem of Richardson and proving a Mumford-type result about topological singularities; this resolves conjectures of Florentino-Lawton.  In particular, we  compute the codimension of the orbifold singular locus using facts about Borel-de Siebenthal subgroups.  We then use the codimension bound to calculate higher homotopy groups of the smooth locus of X, proving conjectures of Florentino-Lawton-Ramras.  Lastly, using the earlier analysis of Borel-de Siebenthal subgroups, we prove a conjecture of Sikora about centralizers of irreducible representations in Lie groups. 

\maketitle

\tableofcontents

\section{Introduction}
Given a finite rank ($r\geqs 1$) free group $\F_r$ and a connected reductive complex affine algebraic group $G$, the {\it $G$-character variety of $\F_r$} is defined as the Geometric Invariant Theory (GIT) quotient of the variety of group homomorphisms $\hom(\F_r,G)$ by the conjugation action of $G$.  Denote this quotient variety by $\X_r(G):=\hom(\F_r,G)\aq G$.   In \cite{FLR}, Florentino-Lawton-Ramras initiated a systematic study, kin to \cite{BGG}, of the homotopy groups of $\X_r(G)$, its GIT stable locus, and also its smooth locus $\mathcal{X}_r(G)$.  Among many theorems in \cite{FLR}, it is proved that $\pi_2(\mathcal{X}_r(G))\cong\Z/n \Z$ when $G$ is the general or special linear group of complex $n\times n$ matrices.  It was then conjectured that this result should generalize to: $$\pi_2(\mathcal{X}_r(G))\cong\pi_1(PG),$$ where $PG$ is the quotient of $G$ by its center.  In this paper, we prove this conjecture and go further and compute $\pi_k(\mathcal{X}_r(G))$ for $0\leqs k\leqs 4$.  Additionally, we obtain periodicity-type results in the higher homotopy groups for the classical groups $G$ (types $A_n$, $B_n$, $C_n$, $D_n$) and some additional higher homotopy groups for the exceptional groups $G_2$, $F_4$, and $E_6, E_7, E_8$.  Surprisingly, we show in a stable range, that $$\pi_k(\mathcal{X}_r(G))\cong \pi_k(G)^r\times \pi_{k-1}(PG).$$  These results comprise Theorem \ref{splitting} and its corollaries.

The proofs of the above theorems rely on a close analysis of the singular locus of these moduli spaces, and in particular, the orbifold singularities.  As $\hom(\F_r,G)$ is naturally identified with the Cartesian product $G^r$, work of Richardson \cite[Theorem 8.9]{Ri} completely describes the algebraic singularities (and by \cite{Milnor} also the $C^k$-singularities for all $k\geqs1$) when $G$ is semisimple, each simple factor in the Lie algebra of $G$ has rank 2 or more, and $r\geqs 2$.  In \cite[Theorem 7.4]{FLR}, Richardson's theorem was generalized to the case when $G$ is connected and reductive.  We extend this further here (Theorem \ref{conj-thm}) by removing the assumption that each simple factor in the Lie algebra of $G$ has rank 2 or more under the stronger condition that $r\geqs 3$ (resolving conjectures in \cite{FL2, FLR}). 

In \cite{Mumford}, Mumford shows if a complex algebraic surface is normal then all singularities are topological singularities (every neighborhood of a singularity in the analytic topology is not homeomorphic to a ball).  We show a ``Mumford-type'' result (Theorems \ref{redugly-thm} and \ref{badugly-thm}) for the class of varieties $\X_r(G)$ (which are all normal), namely, that if $r\geqs 3$, or if $r\geqs 2$ and each simple factor in the Lie algebra of $G$ has rank 2 or more, then all singularities in $\X_r(G)$ are topological singularities (resolving a conjecture in \cite{FL2}). 

Our classification of orbifold singularities and the resulting bounds on codimension
(Theorems \ref{codimbad} and \ref{gencodim}) build on the fundamental work in the thesis of Gu\'erin \cite{Gu1,Gu2}.  In particular, we study Borel-de Siebenthal subgroups, which are, loosely speaking, proper subgroups of a Lie group whose maximal tori are in fact maximal tori of the ambient group (see Definition \ref{def:BdS}).  

Using our results on orbifold singularities, we settle the following conjecture of Sikora as well. A complex reductive Lie group $G$ has property CI if the centralizer of every irreducible subgroup of $G$ coincides with the center of $G$.  In \cite{Si4} it is asked whether the special linear groups are the only simple CI groups.  We answer this question affirmatively with Theorem \ref{CI-thm}. 

\section*{Acknowledgments}
Lawton and Ramras were supported by Collaboration Grants from the Simons Foundation, USA.  We thank IHES for hosting Lawton and Gu\'erin in March 2019 where much of this work was completed.  Gu\'erin was supported by the FNR grant OPEN/16/11405402.  We thank a anonymous referee for very helpful comments and suggestions.

\section{Preliminaries}
Let $\Z$ denote the ring of integers, $\R$ the field of real numbers, and $\C$ the field of complex numbers.

\subsection{Affine Varieties}
By a complex affine variety we mean the zero locus of a collection of polynomials over $\C$. To every affine variety $X$, arising as the zeros of $\{f_1,...,f_n\}\subset\C[x_1,...,x_m]$, there is a {\it coordinate ring} $\C[X]:=\C[x_1,...,x_m]/\langle f_1,...,f_n\rangle$.   Any function $f:X\to Y$ between affine varieties will be called {\it regular} if its component functions are elements in $\C[X]$.

Identifying $X$ with its image in $\C^m$ allows one to give $X$ the subspace topology from the Euclidean metric topology on $\C^m\cong\R^{2m}$.  This topology will be called the {\it analytic} or {\it Euclidean} topology on $X$.  It is not dependent on the choice of presentation of $\C[X]$ up to homeomorphism.  There is also the {\it Zariski} topology on $X$ where closed sets are defined to be the zeros of collections of polynomials in $\C[X]$.  The variety $X$ is irreducible if it is so in the Zariski topology, which implies it is connected in the analytic topology. {\it Throughout this paper we will be considering the analytic topology on complex affine varieties unless stated otherwise.}

$X$ is naturally identified as a set with the spectrum of maximal ideals $\mathrm{Spec}_{max}(\C[X])$ by Hilbert's Nullstellensatz (\cite{HilbertNull}).  Defining a closed set in $\mathrm{Spec}_{max}(\C[X])$ to be the maximal ideals containing a fixed ideal $I$ makes the identification of $X$ with its max-spectrum a homeomorphism (with respect to the Zariski topology). This allows one to create affine varieties by considering finitely generated algebras over fields.  For example, given affine varieties $X,Y$ we can naturally give the Cartesian product $X\times Y$ the structure of an affine variety whose coordinate ring is $\C[X]\otimes_\C\C[Y]$.  

There is a basis for the Zariski topology on $X$ given by the open sets $B_f:=\{x\in X\ |\ f(x)\not= 0\}$ for $f\in \C[X]$.  From this one defines a sheaf of $\C$-algebras on $X$, denoted by $U\mapsto \mathcal{O}_X(U)$, where $\mathcal{O}_X(B_f)=\C[X][1/f]$.  This makes $X$ a locally ringed space.

The dimension of $X$, denoted $\dim_\C(X)$, is the Krull dimension of $\C[X]$, that is, the supremum of the lengths of all chains of prime ideals.  

Given a point $x\in X$, let $\mathfrak{m}_x$ be the corresponding maximal ideal in $\C[X]$.  We can define the Zariski tangent space of $x$ to be the vector space $(\mathfrak{m}_x/(\mathfrak{m}_x)^2)^*$.  A point is {\it smooth} if the dimension of its Zariski tangent space is equal to the dimension of $X$.  In the analytic topology, around each smooth point there is an open set that is holomorphic to an open subset of $\C^d$ where $d=\dim_\C(X)$.  On the other hand, singular points (points that are not smooth) {\it may} admit neighborhoods homeomorphic to Euclidean balls, but such a homeomorphism is never $C^k$-smooth for any $k\geqs 1$ (\cite{Milnor}).  If such a local homeomorphism around a singular point does not exist, we will say that point is a {\it topological} singularity (which we also call {\it ugly}).  

\subsection{Reductive Groups}
A complex affine algebraic group $G$ ($\C$-group for short) is a group that is also a complex affine variety where the group operations are regular functions.  Such a group is {\it reductive} if its radical (the identity component of its maximal normal solvable subgroup) is an algebraic torus (isomorphic to product of $\C^*$'s).  {\it Unless otherwise stated, in this paper $G$ will always be a reductive $\C$-group.}

It is a theorem (\cite{OnVi}) that every reductive $\C$-group is the complexification of a compact Lie group (and vice versa).  Consequently, every reductive $\C$-group has a finite number of components.  When $G$ is further assumed to be connected (as we will do), the central isogeny theorem (\cite{milne}) shows that $G\cong DG\times_F T$, where $DG=[G,G]$ is the derived subgroup, $T$ is a maximal central algebraic torus, and $F=DG\cap T$ is finite.  Since $DG$ is semisimple, its universal cover $\widetilde{DG}$ is a finite product of simply connected simple Lie groups (which are classified by type: $A_n, B_n, C_n, D_n, E_6, E_7, E_8, F_4, G_2$).

A maximal connected solvable subgroup of $G$ is called a {\it Borel} subgroup. A Zariski 
 closed 
 subgroup $P$ is {\it parabolic} if $P$ contains a Borel subgroup.  This is equivalent to $G/P$ being a projective variety (\cite{Borel}).  A {\it Levi} subgroup of a Zariski closed
 subgroup $H\subset G$ is a 
 Zariski closed, 
 connected subgroup $L$ such that $H$ is a semi-direct product of $L$ and the unipotent radical of $H$. Thus all Levi subgroups are reductive. 

\subsection{Geometric Invariant Theory}
General references for Geometric Invariant Theory (GIT) which we refer to are \cite{MFK} and \cite{Do}. Given a reductive $\C$-group $G$ acting on an affine variety $X$ (via regular mappings), there is an induced action on $\C[X]$ via $g\cdot f(x)=f(g^{-1}\cdot x)$.  Any variety with such a $G$-action is called a $G$-variety.  It is a theorem of Nagata (\cite{Na}) that the ring of invariants $\C[X]^G$ is finitely generated since $G$ is reductive (otherwise, it may not be).  Consequently, there is an affine variety $X\aq G:=\mathrm{Spec}_{max}(\C[X]^G)$ called the {\it categorical quotient} of $X$ by $G$ and a regular morphism $\pi_G:X\to X\aq G$ dual to the inclusion $\C[X]^G\hookrightarrow \C[X]$.  It is universal in the sense that every $G$--invariant morphism of varieties $f:X\to Y$
uniquely factors through $\pi_G:X\to X\aq G$.  It is a {\it good} categorical quotient since for every Zariski open $U\subset X\aq G$, we have $\mathcal{O}_{X\aq G}(U)\cong\mathcal{O}_X(\pi_G^{-1}(U))^G$, and disjoint closed $G$-invariant subsets of $X$ map to disjoint closed subsets of $X\aq G$.  Moreover, for every $[x]\in X\aq G$, $\pi_G^{-1}([x])$ contains a unique closed orbit, and $\pi_G(x)=\pi_G(y)$ if and only if the closures of the $G$-orbits of $x$ and $y$ intersect non-trivially.

For a point $x\in X$, let $G(x)$ denote its $G$-orbit and $G_x$ denote its stabilizer.  Let $Z=\cap_{x\in X}G_x$. A point $x\in X$ is called {\it polystable} if $G(x)$ is closed.  Following \cite{Ri}, if $x$ is polystable and $G_x/Z$ is finite the point $x$ is called {\it stable}, and if $G_x=Z$ it is called {\it good}.  The points that are not polystable are called {\it unstable}.  We note that this definition of stable coincides with that of \cite{MFK} for the induced effective action of $G/Z$ on $X$.

Let $X^*$ be the set (or locus) of polystable points, $X^s$ the set of stable points, and $X^g$ the set of good points.  Then $X^*, X^s,$ and $X^g$ are constructible (finite union of locally closed sets), and $X^s$ and $X^g$ are Zariski open (but may be empty).  The action of $G/Z$ on $X^s$ and $X^g$ is proper and so $X^s\aq G\cong X^s/G$ is an orbifold and $X^g\aq G\cong X^g/G$ is a manifold (\cite{JM}).  Any point that is stable but not good will be called {\it bad}; such points are at worst orbifold singular.

By \cite{Lu3}, the orbit space $X^*/G$ is homeomorphic to $X\aq G$ with the analytic topology (see also \cite[Theorem 2.1]{FL4} for an elementary proof).  Moreover, by \cite[Proposition 3.4]{FLR} the non-Hausdorff orbit space $X/G$ is homotopic to $X\aq G$ with the analytic topology.

\subsection{Luna's Slice Theorem}
Let $x\in X^*$, then $G_x$ is a reductive subgroup of $G$.  Let $S_x$ be a $G_x$-invariant locally closed affine subvariety of $X$ containing $x$.  Then there is a morphism $G\times S_x\to X$ given by $(g,s)\mapsto g\cdot s$ and also a $G_x$-action on $G\times S_x$ given by $h\cdot (g,s)=(gh^{-1},h\cdot s)$.  Let $G\times_{G_x}S_x:=(G\times S_x)\aq G_x$.  There is then an induced morphism $\mu:G\times_{G_x}S_x\to X$ given by $[g,s]\mapsto g\cdot s$, and the projection $G\times S_x\to S_x$ induces an epimorphism $G\times_{G_x}S_x\to S_x\aq G_x\cong (G\times_{G_x}S_x)\aq G$. 

We say a morphism $\psi:X\to Y$ is \'etale at $x$ if and only if there is an isomorphism of the completion of local rings: $\widehat{\mathcal{O}_x}\cong\widehat{\mathcal{O}_{\psi(x)}}.$  This implies, in the analytic topology, there is a local homeomorphism at $x$, and at smooth points there is a local diffeomorphism (see \cite{Sch3}).

In \cite{Lu1} (see also \cite{DrJ}) it is proved that for every $x\in X^*$, there exists a $G_x$-invariant locally closed affine subvariety $S_x\subset X$ containing $x$ such that the following diagram commutes:

$$\xymatrix{ G\times_{G_x}S_x \ar[r]^{\mu} \ar[d] & X\ar[d]^{\pi_G} \\ S_x\aq G_x  \ar[r]^{\mu\aq G} & X\aq G,}
$$
where the image of $\mu$ is open, and horizontal morphisms are \'etale.  The subvariety $S_x$ is called an \e{\'etale slice} and this theorem is known as {\it Luna's Slice Theorem}.

Let $N_x:=T_x(X)/T_x(G(x))$. Then $T_x(G(x))$ is invariant under the natural action of $G_x$ on $T_x (X)$, and the induced action of $G_x$ on $N_x:=T_x(X)/T_x(G(x))$ is called the {\it slice representation}. 
Assume that $x$ is a smooth (polystable) point in $X$.  Then there exists a $G_x$-invariant morphism $\varphi:X\to T_x(X)$ that is \'etale at $x$ and $\varphi(x)=0$.  In this case, Luna's Slice Theorem implies there exists a commutative diagram:
$$\xymatrix{ S_x \ar[r]^{\varphi|_{S_x}} \ar[d] & N_x\ar[d]^{\pi_{G_x}} \\ S_x\aq G_x  \ar[r]^{\varphi|_{S_x}\aq G_x\ \ } & N_x\aq G_x,}$$ where the image of $\varphi|_{S_x}$ is open, and horizontal morphisms are \'etale.

Putting these two theorems together, we have:

\begin{prop}\label{sing-princ}

Let $G$ be a connected reductive $\C$-group, $X$ a smooth affine $G$-variety over $\C$, and $x\in X^*$. Then $[x]$ is smooth, singular, or ugly in $X\aq G$ if and only if $[x]$ is respectively smooth, singular, or ugly in $N_x\aq G_x$.
\end{prop}

We will call $N_x\aq G_x$ the {\it local model} of $[x]$ in $X\aq G$.

\section{Character Varieties: Good, Bad, and Ugly}\label{GBU-sec}
Most of the results in this section are taken directly from \cite{Ri}.  Compare also \cite{Si4}.

\subsection{Representations: Irreducible, Reducible, \& Completely Reducible}
Let $G$ be a reductive $\C$-group. A {\it $G$-representation of $\F_r$}, or simply a {\it representation} when the context is clear, will be a group homomorphism from a rank $r$ free group $\F_r$ to $G$.  We refer to these homomorphisms as representations since there exists (\cite[Theorem 8 on p. 104]{OnVi}) a faithful algebraic morphism $G\hookrightarrow \mathrm{GL}_n(\C)$ for some natural number $n$ (which can be assumed minimal), and so each $\rho\in \hom(\F_r,G)$ is a representation of $\F_r$ in the traditional sense (although not canonically).

Since $\hom(\F_r,G)$ is naturally identified with the Cartesian product $G^r$, and $G$ is an affine variety, $\hom(\F_r,G)$ can be given the structure of an affine variety too; it is necessarily irreducible if $G$ is connected. As $G$ is a Lie group, $\hom(\F_r,G)$ is smooth and has the structure of a (holomorphic) manifold in the analytic topology.

The natural equivalence for representations is conjugation.  Precisely, $G$ acts rationally on $\hom(\F_r,G)$ by $g\cdot \rho(w)=g\rho(w)g^{-1}$ for all $w\in \F_r$.   Let $Z(G)$ be the center of $G$, and $PG:=G/Z(G)$.  Then $PG$ acts on $\hom(\F_r,G)$ by conjugation as well since the center acts trivially.

A representation $\rho\in\hom(\F_r, G)$ is \e{irreducible} if the image $\rho(\F_r)$ is not contained in a proper parabolic subgroup of $G$. By \cite[Theorem 4.1]{Ri}, a representation $\rho \in \hom(\F_r, G)$ is irreducible if and only if it is stable.

Since the stable locus in GIT is Zariski open and $\pi_G$-saturated, the set of irreducible representations $\hom(\F_r,G)^{irr} \subset \hom(\F_r, G)\isom G^r$ is as well, and 
hence $\hom(\F_r,G)^{irr}$ is a smooth submanifold of $G^r$.
Consequently, the complement of the irreducible locus, denoted $\hom(\F_r,G)^{red}$ and called the set of {\it reducible} representations, is Zariski closed and $\pi_G$-saturated (thus, $\hom(\F_r,G)^{red}$ is a finite union of locally closed submanifolds).  Moreover, by \cite[Lemma 3.3]{Ri}, the irreducible locus is non-empty if $r\geqs 2$ (and empty if $r=1$ and $G$ is non-abelian by \cite[Remark(b) to Theorem 4.1]{Ri}).

Since there are only finitely many conjugacy classes of proper parabolic subgroups of $G$, the reducible locus is the union of finitely many algebraic sets of the form $$H_P=\cup_{g\in G/P} \hom(\F_r,gPg^{-1})$$ where $P < G$ is a proper parabolic. Since $G/P$ is complete,  $H_P$ is Zariski closed; see \cite[Proposition 27]{Si4}. 

A representation $\rho$ is {\it completely reducible} if for every proper parabolic $P$ containing $\rho(\F_r)$, there is a Levi subgroup $L<P$ with $\rho(\F_r)< L$.  Note that irreducible representations are completely reducible (vacuously).  By \cite[Theorem 5.2]{Ri}, a representation $\rho \co \F_r \to G$ is completely reducible if and only if it is polystable.  

The GIT quotient $\mathfrak{X}_r(G):=\hom(\F_r,G)\aq G$ is called the {\it $G$-character variety of $\F_r$}.  For any labeled subset $\hom(\F_r,G)^\textsc{label}\subset \hom(\F_r,G)$, we will let $\X_r(G)^\textsc{label}:=\pi_G(\hom(\F_r,G)^\textsc{label})$. So, for example, $\XC{r}(G)^{irr}$ is Zariski open and $\XC{r}(G)^{red}$ is Zariski closed.

\subsection{Representations: Good \& Zariski Dense}

Following  \cite{JM}, we define the {\it good locus} to be the subspace $$\hom(\F_r,G)^{good}\subset\hom(\F_r,G)^{irr}$$
of representations whose $PG$-stabilizer is trivial.  Representations in the good locus will be called {\it good}.  
By \cite[Proposition 1.3]{JM}, $\hom(\F_r,G)^{good}$  is Zariski open, and since it is $\pi_G$--saturated, it is Zariski open as well. In particular, $\hom(\F_r,G)^{good}$ is an open subset and a smooth submanifold of $G^r$ (non-empty for $r\geqs 2$). 

There is another locus that is important to consider.  Let $\hom(\F_r,G)^{zd}$ be the set of representations whose image is Zariski dense.  This set is contained in the good locus, but is generally not equal to it.  For example, if one takes the principal $\SL_2(\C)$ inside some complex reductive group $G$ (different from type $A_1$), then any representation that is Zariski-dense in $\SL_2(\C)$ will induce $($with this principal inclusion$)$ a good representation in $G$ whose image is not Zariski dense (see \cite[Chapter VIII, Paragraph 11, Section 4]{Bour-lie} for the existence of the principal $\SL_2(\C)$ and \cite{Tit} for the fact that it only commutes with central elements in $G$).

To analyze the Zariski dense locus, we will need the following fact regarding Zariski closed subgroups.

\begin{lem}\label{Goursat}
Let $G$ be a semisimple $\C$-group and let $T$ be a complex algebraic torus. If $H$ is a subgroup of $G\times T$ then $H$ is Zariski dense if and only if the projections of $H$ onto the two factors are Zariski dense. 
\end{lem}
\begin{proof}
Let $p_1:G\times T\to G$ and $p_2:G\times T\to T$ be the projections onto the factors of $G\times T$.

First assume that $H$ is Zariski dense in $G\times T$.  Then by continuity of $p_1$ and $p_2$, $p_1(H)$ and $p_2(H)$ are both Zariski dense.

Now assume that $p_1(H)$ and $p_2(H)$ are Zariski dense in $G$ and $T$, respectively. Let $K$ denote the Zariski closure of $H$. Then $K$ is a subgroup of $G\times T$, and we claim that $K$ surjects onto $G$ and $T$ via the corresponding projections. Indeed, since $G$ and $T$ are compact in the Zariski topology, the projections are closed maps (with respect to the Zariski topology) and hence $p_i (K)$ is a Zariski closed set containing $p_i (H)$.

By Goursat's Lemma (see \cite[Sections 11-12]{Gou} or \cite[Section 2]{Lam}), $\mathrm{Ker}(p_1)\cap K$ is normal in $T$, $\mathrm{Ker}(p_2)\cap K$ is normal in $G$, and $K$ is identified through $p_1\times p_2$ as the graph of an isomorphism between $G/(\mathrm{Ker}(p_2)\cap K)$ and  $T/(\mathrm{Ker}(p_1)\cap K)$. Since $G/(\mathrm{Ker}(p_2)\cap K)$ is a semisimple Lie group and $T/(\mathrm{Ker}(p_1)\cap K)$ is abelian, both of these (isomorphic) quotients must be trivial. 

As a result, $G=\mathrm{Ker}(p_2)\cap K$ and $T=\mathrm{Ker}(p_1)\cap K$. It follows that $K$ contains both factors $G$ and $T$, whence $K=G\times T$. 
\end{proof}

\begin{prop}\label{zd-prop}
If $r\geqs 2$ and $G$ is a connected, reductive $\C$-group, then $\X_r(G)^{zd}:=\hom(\F_r,G)^{zd}/G$ is Euclidean dense in $\X_r(G)$.
\end{prop}

\begin{proof}
When $r\geqs 2$ \cite[Lemma 1.8]{Gelander} implies $\hom^{zd}(\F_r,G)\subset \hom(\F_r,G)$ is non-empty.

For our first case, suppose $G$ is semisimple. Then $\hom^{zd}(\F_r,G)$ is Zariski open by \cite[Proposition 8.2]{AcBu} and hence dense since $\hom(\F_r,G)$ is connected.  Let $$\pi_G:\hom(\F_r, G)\to \X_r( G)$$ be the GIT quotient map. Every Zariski dense $\rho$ is irreducible and, hence, a stable point of the action of $G$ by conjugation, by \cite[Corollary 31]{Si4}. By \cite{Do} (Theorem 8.1 and Property (iv) of a good categorical quotient), $\pi_G^{-1}([\rho])$ is the (set-theoretic) $G$-orbit of $\rho$.  Thus, since $\hom(\F_r,G)^{zd}$ is invariant under the conjugation action, we conclude that $\pi_G^{-1}(\pi_G(\hom(\F_r,G)^{zd}))=\hom(\F_r,G)^{zd}$.  Thus, $\X^{zd}_r(G)$ is Zariski open and hence dense in $\X_r(G)$.

Now assume $G$ is connected and reductive. Then $G\cong DG\times_F T$ where $DG=[G,G]$ is semisimple, $T$ is a central algebraic torus, and $F=T\cap DG$ is a finite central subgroup. From this, by \cite[Proposition 5]{Si7} or \cite{BLR}, the inclusions of $DG$ and $T$ into $G$ induce an isomorphism $\varphi: \X_r(DG)\times_{F^r}\X_r (T) \to \X_r(G)$. 

Lemma~\ref{Goursat} implies that $\varphi$ maps the subspace
$\X_r(DG)^{zd} \times_{F^r}\X_r (T)^{zd} \subset \X_r(DG)\times_{F^r}\X_r (T)$ into $\X_r (G)^{zd}$, so to complete the proof it will suffice to show that $\X_r(DG)^{zd} \times_{F^r}\X_r (T)^{zd}$ is Euclidean dense in $\X_r(DG) \times_{F^r} \X_r (T)$, and for this it suffices to show that $\X_r(DG)^{zd}$ and $\X_r (T)^{zd}$ are Euclidean dense in $\X_r(DG)$ and $\X_r (T)$, respectively.

We have seen above that $\X_r(DG)^{zd}$ is non-empty and also Zariski open in $\X_r(DG)$.
Since $\X_r(DG)$ is irreducible, it follows that  $\X_r(DG)^{zd}$ is Euclidean dense in $\X_r(DG)$. Since the set of elements of $T^r$ generating Zariski dense subgroups in $T$ is Euclidean dense, $\X_r(T)^{zd}$ is dense in $\X_r (T) = T^r$, completing the proof.
\end{proof}

\begin{rem}
Proposition \ref{zd-prop} and its proof are a special case of \cite[Lemma 6]{LS-v3}.  Lawton notes that there is a minor mistake in the published version of \cite{LS} that does not change any of the main results and is corrected in the arXiv update \cite{LS-v3}.  Precisely, in Theorem 6, and Corollaries 7, 8, and 10 in \cite{LS}, ``irreducible'' needs to be replaced by ``Zariski-dense.''
\end{rem}

\subsection{Representations: Bad \& Ugly}

Define the locus of {\it bad} representations as the collection of irreducible representations that are not good; that is, whose $G$-stabilizer is strictly larger than the center of $G$. 

The collection of bad representations, denoted $\hom(\F_r,G)^{bad}$, is thus Zariski closed in the irreducible locus (again this follows from \cite[Proposition 1.3]{JM}).

We denote the smooth locus $\mathcal{X}_r(G):=\XC{r}(G)-\XC{r}(G)^{sing}$, where $\XC{r}(G)^{sing}$ is the subvariety of (algebraic) singular points.

\begin{rem}
From \cite{FL2}, if $G$ is $\SL_n(\C)$ or $\GL_n(\C)$ then $$\mathcal{X}_r(G)=\XC{r}(G)^{good}=\XC{r}(G)^{irr}$$ as long as $(r-1)(n-1)\geqs 2$. From \cite{HP}, this result does hold true when $G=\p\SL_2(\C)$.  In Section \ref{Schur}, we show that if $G$ is simple the only time the bad locus is empty is when $G=\SL_n(\C)$, which resolves a conjecture of Sikora \cite{Si4}.
\end{rem}

Define an {\it ugly representation} to be one that is a topological singularity in $\X_r(G)$ with respect to the analytic topology. 

In other words, $\rho$ is ugly if and only if every open set around $[\rho]$ is {\it not} homeomorphic to a Euclidean space. When $G$ is connected we can say more: if $[\rho]$ is not ugly, then it has a neighborhood homeomorphic to $\C^d$, where $d=\dim_\C\X_r(G)$. This follows from  Invariance of Domain together with density of the smooth locus.
 
We note here that: $$\dim_\C\X_r(G)=(r-1)\dim_\C G+\dim_\C Z(G)$$ if $r\geqs 2$, and $\dim_\C\X_1(G)=\mathrm{Rank}(G)$, where the rank of $G$ is the dimension of one of its Cartan subgroups (the centralizer of a maximal torus).  In \cite{FL2}, it is shown when $G$ is $\GL_n(\C)$ or $\SL_n(\C)$ that there exists ugly representations if and only if  $(r-1)(n-1)\geqs 2$; in these cases all ugly representations are in the reducible locus.

\subsection{Summary}
In \cite{Ri}, Richardson showed that the singular locus of $\X_r(G)$ is precisely the union of the bad locus with the reducible locus if $G$ is semisimple, $r\geqs 2$ and the Lie algebra of $G$ does not have any rank 1 simple factors.  In \cite{FLR} this result is extended to connected reductive groups $G$ with the same condition on the simple factors of the Lie algebra of the derived subgroup $DG$.

In Section \ref{singsec}, we will establish that if $r\geqs 3$, then the singular locus coincides with the union of the bad locus and the reducible locus, and we will show that all algebraic singularities are in fact topological singularities (ugly).  The first result extends Richardson's Theorem by removing the local rank condition when $r\geqs 3$. The second is analogous to Mumford's result that all algebraic singularities in a normal complex algebraic surface are ugly.

In Figure \ref{gbrdiagram}, we have organized the previously defined subspaces of representations in a Venn diagram.

\begin{figure}[!h]
\includegraphics[width=\linewidth]{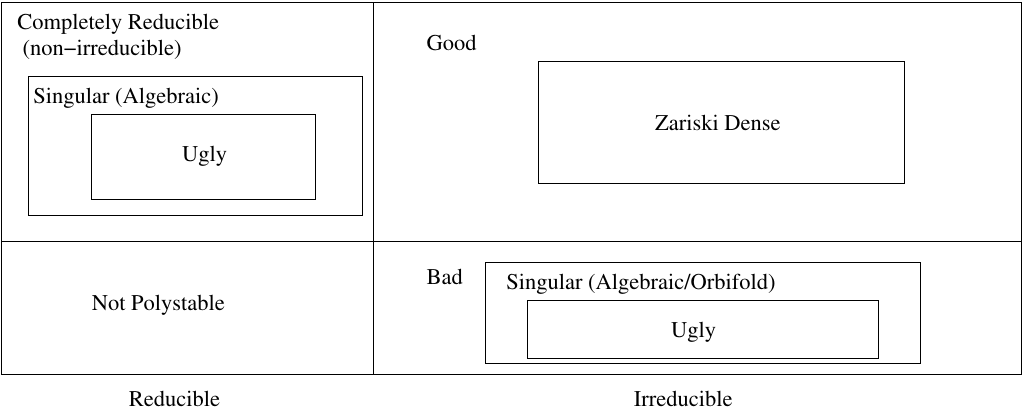}
\caption{Venn Diagram of Representations.}\label{gbrdiagram}
\end{figure}

\section{Bounding the codimension of bad representations}

To prove our main theorems, we need to bound the codimension of the bad locus.  To do so, we need to study the Lie algebras of reductive $\C$-groups.  One can find the following terms, facts and notation in any standard text covering Lie algebras like \cite{OnVi}, or \cite{FH}. 

A complex Lie algebra $\mathfrak{g}$ is {\it reductive} if its radical (largest solvable ideal) is equal to its center.  There are many equivalent formulations. In particular, $\mathfrak{g}$ is a direct sum of its center and a finite number of simple subalgebras (not containing any non-trivial proper ideals).  If $G$ is a reductive $\C$-group, then its Lie algebra, denoted $\mathrm{Lie} (G)$ or $\mathfrak{g}$, is a complex reductive Lie algebra.

A {\it Cartan subalgebra} $\mathfrak{h}\subset \mathfrak{g}$ is a maximal abelian subalgebra such that $\mathrm{ad}_H$ is diagonalizable for each $H\in \mathfrak{h}$.  Given a Cartan subalgebra, a {\it root} (or more specifically, an  $\mathfrak{h}$-root) is a non-zero element $\alpha\in \mathfrak{h}^*$, the dual of $\mathfrak{h}$, for which there exists a non-zero $X\in \mathfrak{g}$ with $\alpha (H)X=[H,X]=\mathrm{ad}_H(X)$ for all $H\in \mathfrak{h}$. We denote the set of roots by $\Delta$. So for each $\alpha\in \Delta$ we have a generalized $\mathrm{ad}_{\mathfrak{h}}$-eigenspace (called the \e{root space})
$$\mathfrak{g}_\alpha = \{X\in \mathfrak{g} :\, \mathrm{ad}_H(X) = \alpha (H) X \textrm{ for all } H\in \mathfrak{h}\}.$$  
It is a standard fact that each $\mathfrak{g}_\alpha$ is one-dimensional (over $\C$), and we have the \e{root space decomposition}
\[\mathfrak{g} = \mathfrak{h}\oplus \bigoplus_{\alpha\in \Delta}\mathfrak{g}_{\alpha}.\]

The following definition, which can be found in \cite{Ruben}, will be important for our purposes.

\begin{defn}
Let $\mathfrak{g}$ be a Lie algebra. We say that a subalgebra $\mathfrak{s}\subset \mathfrak{g}$ is {\bf regular}
 if it contains a Cartan subalgebra  of $\mathfrak{g}$. 
\end{defn}

\begin{lem}\label{regular}
Let $\mathfrak{s}\subset \mathfrak{g}$ be a regular subalgebra of the complex Lie algebra $\mathfrak{g}$, and let $\mathfrak{h}$ be a Cartan subalgebra of $\mathfrak{g}$ that is contained in $\mathfrak{s}$. Then 
\[\mathfrak{s} = \mathfrak{h}\oplus \bigoplus_{\alpha\in \Delta'}\mathfrak{g}_{\alpha},\]
where $ \Delta' \subset \Delta$ is the set of $\mathfrak{h}$--roots satisfying 
$\mathfrak{g}_\alpha\cap \mathfrak{s}\neq 0$.
\end{lem}

Note that since $\dim_\C \mathfrak{g}_\alpha = 1$, the condition $\mathfrak{g}_\alpha\cap \mathfrak{s}\neq 0$ is equivalent to $ \mathfrak{g}_\alpha \subset \mathfrak{s}$.

\begin{proof} Since $\mathfrak{h} \subset \mathfrak{s}$ and each element of $\mathfrak{g}$ can be written in the form $X = H + \sum_{\alpha\in \Delta} X_\alpha$ (with $H\in \mathfrak{h}$ and $X_\alpha\in \mathfrak{g}_\alpha$), it suffices to check that whenever such an element $X$ lies in $ \mathfrak{s}$, each term $X_\alpha$ actually lies in  $\mathfrak{s}$. Moreover, since $\mathfrak{h}\subset\mathfrak{s}$, it suffices to prove   the statement: if $X = \sum_{\alpha\in \Delta} X_\alpha \in \mathfrak{s}$, for some $X_\alpha \in \mathfrak{g}_\alpha$, then $X_\alpha\in \mathfrak{s}$ for each $\alpha$. 
For any such vector $X$, let $n(X)$ denote the number of non-zero terms $X_\alpha$. 
We will prove the statement by induction on $n(X)$.

For the base case, $n(X) = 1$, there is nothing to prove. 
For the induction step, let $\{\alpha_1, \ldots, \alpha_k\}$ be a basis for $\mathrm{Span}_\C \{\alpha:\, X_\alpha\neq 0\}$.

If $k=1$, then writing $\alpha = \alpha_1$, either $X = X_{\alpha}$ and there is nothing to prove, or else $-\alpha \in \Delta$ and $X = X_{\alpha} + X_{-\alpha}$ for some $X_{-\alpha} \in \mathfrak{g}_{-\alpha}$. In this case, choose $H\in \mathfrak{h}$ with $\alpha(H)\neq 0$. Then
$$[H, X] = \alpha(H) X_\alpha - \alpha(H) X_{-\alpha} = \alpha(H) (X_\alpha - X_{-\alpha}) \in 
\mathfrak{s},$$
so $X_\alpha - X_{-\alpha} \in \mathfrak{s}$ as well, and after adding $X$ we see that $X_\alpha, X_{-\alpha} \in  \mathfrak{s}$, as desired.

If $k\geqs 2$,  
choose an element $H_1\in \mathfrak{h}$ such that $\alpha_i (H_1)$ is non-zero if and only if $i=1$. Then setting $X' = X - \sum_i X_{\alpha_i}$, we have
$$[H_1, X] = \alpha_1 (H_1) X_{\alpha_1} + [H_1, X'] \in \mathfrak{s}.$$
We can apply our induction hypothesis to the element $\alpha_1 (H_1) X_{\alpha_1} + [H_1, X']$, so we conclude that $X_{\alpha_1} \in \mathfrak{s}$. Applying the induction hypothesis to $X - X_{\alpha_1}$, we find that all the remaining terms in $X$ also lie in $\mathfrak{s}$.

\end{proof}

\begin{lem}\label{regular2}
Let $\mathfrak{z}, \mathfrak{g}_1, \ldots, \mathfrak{g}_n$ be complex Lie algebras, with $\mathfrak{z}$ abelian and $\mathfrak{g}_i$ simple, and let $\mathfrak{g} =\mathfrak{z} \oplus \mathfrak{g}_1\oplus \cdots \oplus \mathfrak{g}_n$. Then every regular subalgebra of  $\mathfrak{g}$ is conjugate to one of the form $\mathfrak{z}  \oplus \mathfrak{s}_1\oplus \cdots \oplus \mathfrak{s}_n$, where $\mathfrak{s}_i$ is a regular subalgebra of $\mathfrak{g}_i$.
\end{lem}
\begin{proof} Letting $\mathfrak{h}_i$ be a Cartan subalgebra of $\mathfrak{g}_i$. Then 
$\mathfrak{h}:=\mathfrak{z} \oplus \mathfrak{h}_1\oplus \cdots \oplus \mathfrak{h}_n$ is a Cartan subalgebra of $\mathfrak{g}$, and since all Cartan subalgebras of $\mathfrak{g}$ are conjugate, it suffices to show that every subalgebra of $\mathfrak{g}$ containing $\mathfrak{h}$ has the desired form. 

So, we now  assume that $\mathfrak{s}$ contains $\mathfrak{h}$. 
Let $\Delta$ be the set of  $\mathfrak{h}$--roots of $\mathfrak{g}$.
By Lemma~\ref{regular},  there exists $\Delta' \subset \Delta$ such that  
\begin{equation}\label{s-eq}\mathfrak{s} = \mathfrak{h}\oplus \bigoplus_{\alpha\in \Delta'}\mathfrak{g}_{\alpha}\end{equation}

We claim that each of the generalized eigenspaces $\mathfrak{g}_{\alpha}$, $\alpha\in \Delta$, is in fact 
contained in $\mathfrak{g}_i$ for some $i$. 
 Let $\pi_i \co \mathfrak{h}\to \mathfrak{h}_i$ and $j_i\co \mathfrak{g}_i\injects \mathfrak{g}$ be the projection and inclusion maps, respectively, for the $i$th factors, and let $\Delta_i$ be the set of $\mathfrak{h}_i$--roots of $\mathfrak{g}_i$. Computation shows that $\pi_i^* (\Delta_i) \in \Delta$ -- more specifically,  for all $\alpha_i\in \Delta_i$ we have $j_i \left((\mathfrak{g}_i)_{\alpha_i}\right) \subset \mathfrak{g}_{\pi_i^* \alpha}$, and in fact equality holds since these root spaces are one-dimensional. Comparing the root space decompositions of $\mathfrak{g}$ and of the $\mathfrak{g}_i$'s, we see that $\Delta = \bigcup_i (\pi_i^* (\Delta_i))$. Since $\mathfrak{g}_{\pi_i^* \alpha} = j_i \left((\mathfrak{g}_i)_{\alpha_i}\right)$, this proves the claim.

It follows that in the decomposition (\ref{s-eq}), each summand (apart from $\mathfrak{h}$) is contained in $j_i (\mathfrak{g}_i)$ for some $i$. Letting $\mathfrak{s}_i$ be the direct sum of $\mathfrak{h}_i$ with those summands contained in $j_i (\mathfrak{g}_i)$, we obtain the desired decomposition of $\mathfrak{s}$.
\end{proof}

We note, as the reader may be wondering, that regular elements in a Lie algebra (elements with minimal dimensional centralizers) are not directly related to the regular subalgebras as defined above.  However, a key example of a regular subalgebra is the centralizer of a semisimple element (an element that is diagonalizable with respect to a finite dimensional representation of $G$).

\begin{lem}\label{semisimplecentralizers}  Let $G$ be a connected, reductive $\C$-group with Lie algebra $\mathfrak{g}$. If $g\in G$ is a non-central semisimple element, then the centralizer $$\mathfrak{z}_{\mathfrak{g}}(g):=\{X\in\mathfrak{g}\ |\ \mathrm{Ad}_g(X)=X\}$$ is a regular reductive subalgebra of $\mathfrak{g}$.
\end{lem}

\begin{proof}
A Cartan subgroup of $G$ is a subgroup such that its Lie algebra is Cartan.  Since $g$ is semisimple, there exists a Cartan subgroup containing $g$.  Now let $\mathfrak{h}$ be the corresponding Cartan subalgebra, and since $\mathfrak{g}$ is reductive, we know $\mathfrak{g}=\mathfrak{h}\oplus \bigoplus_{\alpha\in \Delta}\mathfrak{g}_{\alpha}$ where $\Delta$ is the set of roots with respect to $\mathfrak{h}$.  Since the derivative of $\mathrm{Ad}$ is $\mathrm{ad}$, we have $\mathfrak{h}\subset \mathfrak{z}_{\mathfrak{g}}(g)$.
Now let $\Delta'\subset \Delta$ be the $\mathfrak{h}$-roots $\alpha$ such that $\mathfrak{z}_{\mathfrak{g}}(g)\cap \mathfrak{g}_\alpha\not=0$.  
Then by Lemma~\ref{regular}, 

$$\mathfrak{z}_{\mathfrak{g}}(g)=\mathfrak{h}\oplus \bigoplus_{\alpha\in \Delta'}\mathfrak{g}_{\alpha}.$$

Since the center of $\mathfrak{g}$ is contained in $\mathfrak{h}$ and $\mathfrak{z}_{\mathfrak{g}}(g)$ contains $\mathfrak{h}$, the center of $\mathfrak{z}_{\mathfrak{g}}(g)$ is equal to the center of $\mathfrak{g}$.  So to show $\mathfrak{z}_{\mathfrak{g}}(g)$ is reductive it suffices to show that its quotient by the center of $\mathfrak{g}$ is a semisimple Lie algebra.  This latter fact is equivalent to $\Delta'$ being stable by multiplication by $-1$.  
For each $\alpha\in \Delta\cup \{0\} \subset \mathfrak{h}^*$, set $\lambda_\alpha = e^{\alpha (H_0)}$.  Then $\mathrm{Ad}_g(X)=\lambda_\alpha X$, and $\lambda_{\alpha+\beta}=\lambda_\alpha\lambda_\beta$.
Moreover, $\lambda_0 = 1$, and hence for each $\alpha \in \Delta$ we have
 $1=\lambda_0=\lambda_\alpha\lambda_{-\alpha}$.  Now, $\alpha\in \Delta'$ if and only if $\lambda_\alpha=1$, so we conclude that if $\alpha\in \Delta'$ then $-\alpha\in \Delta'$ too.  Thus, $\Delta'$ is stable under multiplication by $-1$ and so $\mathfrak{z}_{\mathfrak{g}}(g)$ is reductive.
\end{proof}

We will call a subgroup of $G$ {\it bad} if it is not contained in a parabolic subgroup and its centralizer is not equal to $Z(G)$ (see \cite{Gu1}, for example).  Thus, a $G$-representation of $\F_r$ is bad if and only if its image is a bad subgroup.

To understand bad subgroups we will need to understand {\it maximal} regular subalgebras.  Such a subalgebra  is either parabolic (contains a Borel subalgebra, i.e. a maximal solvable subalgebra) or reductive.  The latter case leads to the following definition (see \cite{Ruben, BdS}). 

\begin{defn}\label{def:BdS}
Let $G$ be a reductive $\C$-group and let $\mathfrak{g}$ be a reductive Lie algebra. We say that a reductive subalgebra $\mathfrak{s}$ of $\mathfrak{g}$ is a {\bf Borel-de Siebenthal} subalgebra if it is proper, has the same rank as $\mathfrak{g}$ and $\mathfrak{s}/\mathfrak{z}(\mathfrak{g})$ is semisimple.\footnote{Note that every subalgebra of full rank must contain the center.}  A subgroup of $G$ is said to be a {\bf Borel-de Siebenthal subgroup} of $G$ if it is connected and its Lie algebra is a Borel-de Siebenthal subalgebra of $\mathrm{Lie} (G)$. 
\end{defn}

One easily shows that every Borel-de Siebenthal subalgebra is regular.   When $\mathfrak{g}$ is semisimple, a Borel-de Siebenthal subalgebra is a semisimple proper subalgebra of $\mathfrak{g}$ of maximal rank.  In the sequel we will write {\it BdS} to abbreviate ``Borel-de Siebenthal.''

\begin{exam}
The subalgebra $\mathfrak{so}_{2n}(\mathbb{C})$ inside $\mathfrak{so}_{2n+1}(\mathbb{C})$ is a BdS subalgebra as they both have rank $n$.  The subalgebras $\mathfrak{sp}_{2k}(\mathbb{C})\times \mathfrak{sp}_{2n-2k}(\mathbb{C})$ inside $\mathfrak{sp}_{2n}(\mathbb{C})$ are  BdS subalgebras for $1\leqs k\leqs n-1$ since the rank of the former is $k+(n-k)=n$ which is the rank of the latter. 
\end{exam}

\begin{rem}
Parabolic subalgebras of $\mathfrak{g}$, up to conjugation, are in bijection with subsets of nodes of the Dynkin diagram of $\mathfrak{g}$, see \cite{Borel}.  So maximal parabolic subalgebras are classified up to conjugation by removing a single vertex from the Dynkin diagram of the original Lie algebra.  While it is not trivial, it is also possible to classify BdS subalgebras up to conjugation. Using the isomorphism \[\mathfrak{z}(\mathfrak{g})\times [\mathfrak{g},\mathfrak{g}]\to \mathfrak{g}\] and the fact that semisimple algebras are direct sums of simple ones, it suffices to do it in the simple case. Then, the classification comes down to classifying sub-root systems of the root system of $\mathfrak{g}$ which have the same rank as $\mathfrak{g}$. Such an analysis leads to Table \ref{BdS} of maximal BdS subalgebras. 
\end{rem}

For $\rho:\F_r\to G$ we let $Z_G(\rho)$ be the centralizer of the image of $\rho$.  In these terms, $\rho:\F_r\to G$ is a bad representation if and only if $\rho$ is irreducible and $Z_G(\rho)/Z(G)$ is not trivial. 

\begin{prop}\label{commsemi}
Let $\rho$ be an irreducible representation, then $Z_G(\rho)/Z(G)$  is finite.  Consequently, if $g\in G$ commutes with an irreducible representation it must be semisimple.
\end{prop}
\begin{proof}
Corollary 17 in \cite{Si4} says $Z_G(\rho)/Z(G)$ is finite in slightly different language.  Thus, if $g\in G$ commutes with $\rho$ then, $gZ(G)$ has finite order in $Z_G(\rho)/Z(G)$, and thus there exists $n$ so $g^n$ is central.  Since $Z(G) \isom (\mathbb{C}^*)^k \cross A$ for some finite abelian group $A$, we can write $g^n = s^n a$ for some $s\in (\mathbb{C}^*)^k$, $a\in A$, and now $g = (gs^{-1}) s$ is a product of a finite order (hence semisimple) element with a central element.  Thus, $g$ is semisimple.
\end{proof}

Proposition \ref{commsemi} and Lemma \ref{semisimplecentralizers} together allow us to prove the following characterization of bad representations.

\begin{prop}\label{LBDS}
Let $G$ be a connected, reductive $\C$-group. Then the image of a bad representation $\rho:\F_r\to G$ is contained in the normalizer of some Levi subgroup of a proper parabolic subgroup or of some BdS subgroup of $G$.
\end{prop}

\begin{proof}
Let $B=\rho(\F_r)$ be a bad subgroup of $G$.  Since $B$ is bad, there exists a non-central element $\xi$ commuting with every element of $B$. Proposition \ref{commsemi} implies that $\xi$ is semisimple. 

Let $Z:=Z_G(\xi)^0$ be the identity component of $Z_G(\xi)=\{g\in G\ |\ g\xi g^{-1}=\xi\}$.  Let $N$ be the normalizer of $Z$. Then since $bzb^{-1}\xi=\xi bzb^{-1}$ for all $z\in Z$ and all $b\in B$, and conjugation preserves identity components, we conclude $B\leq N$.

By Lemma \ref{semisimplecentralizers}, the Lie algebra  $\mathfrak{z}$  of $Z$ is regular and reductive. If $\mathfrak{z}/\mathfrak{z}(\mathfrak{g})$ is semisimple then $Z$ is a BdS subgroup by definition and we are done. 

Otherwise, let $\mathfrak{t}$ be the inverse image in $Z$ of the radical of $\mathfrak{z}/\mathfrak{z}(\mathfrak{g})$. We claim that  $\mathfrak{t}$ is abelian. Indeed, since  $\mathfrak{z}$ is reductive, so is $\mathfrak{z}/\mathfrak{z}(\mathfrak{g})$, and hence the radical $\mathfrak{t}/\mathfrak{z}(\mathfrak{g})$ is abelian; but then  
$\mathfrak{t} \subset \mathfrak{z}$ is a solvable ideal in $\mathfrak{z}$, so it too is abelian.
Now $\mathfrak{t}$ 
 exponentiates to a torus (i.e. a connected, semisimple, abelian group) in $Z$ which is central in $Z$ but not in $G$. In particular,  $L=Z_G(Z(Z)^0)$ is proper. Since $B$ normalizes $Z$, $B$ normalizes $L$. Since $Z(Z)^0$ is a torus, its centralizer in $G$ is a Levi subgroup of a  parabolic subgroup of $G$, which must be proper since $L$ is proper.
\end{proof}

Recall that a complex Lie algebra $\mathfrak{g}$ is the semi-direct product of its radical and a semisimple subalgebra called a {\it Levi} subalgebra.   Proposition \ref{LBDS} highlights a dichotomy  that motivates the following definition.

\begin{defn}\label{types}
Let $G$ be a reductive $\C$-group, $\mathfrak{g}$ its Lie algebra, and $\rho:\F_r \to G$ a bad representation. We say that $\rho$ is: 
\begin{enumerate}
\item[] {\bf Type 1} if $\rho(\F_r)$ normalizes a Levi subalgebra of a proper parabolic subalgebra, and 
\item[] {\bf Type 2} if $\rho(\F_r)$ normalizes a BdS subalgebra. 
\end{enumerate}
\end{defn}

There is nothing preventing a bad representation to be both of Type 1 and Type 2. When $G$ is simply connected, only bad representations of Type 2 may arise (see Section \ref{Schur} for details). 

For a subalgebra $\mathfrak{s}$ of a Lie algebra $\mathfrak{g}$, let $\mathfrak{n}_{\mathfrak{g}}(\mathfrak{s}):=\{X\in \mathfrak{g}\ |\ [X,Y]\in \mathfrak{s},\text{ for all }Y\in \mathfrak{s}\}$ be the normalizer of $\mathfrak{s}$.

\begin{lem}\label{liealgnorm}
Let $\mathfrak{g}$ be a reductive group and $\mathfrak{s}$ be a regular subalgebra of $\mathfrak{g}$. Then 
\[\mathfrak{n}_{\mathfrak{g}}(\mathfrak{s})=\mathfrak{s}.\]
\end{lem}
\begin{proof}
Clearly, $\mathfrak{s}\subset \mathfrak{n}_{\mathfrak{g}}(\mathfrak{s})$.  Let $\mathfrak{h}$ be a Cartan subalgebra of $\mathfrak{g}$ contained in $\mathfrak{s}$. It follows that $\mathfrak{n}_{\mathfrak{g}}(\mathfrak{s})$ is regular (since it contains $\mathfrak{h}$). By Lemma~\ref{regular} we may write 
\[\mathfrak{h}\subset\mathfrak{n}_{\mathfrak{g}}(\mathfrak{s})=\mathfrak{h}\oplus \bigoplus_{\alpha\in \Delta'}\mathfrak{g}_{\alpha}\]
where $\Delta'$ is the subset the set of $\mathfrak{h}$--roots $\mathfrak{g}_\alpha$ of $\mathfrak{g}$ satisfying $\mathfrak{g}_\alpha\cap\mathfrak{n}_{\mathfrak{g}}(\mathfrak{s})\not=0$. Then, for any $H\in \mathfrak{h} \subset \mathfrak{s}$, $\alpha\in \Delta'$ and non-zero $X_{\alpha}\in \mathfrak{g}_{\alpha}$, since $X_\alpha$ normalizes $\mathfrak{s}$ we have 
\[ [X_{\alpha},H] = \alpha(H)X_{\alpha} \in \mathfrak{s}.\]

If $X_{\alpha}$ does not belong to $\mathfrak{s}$, then we must have $\alpha(H)=0$. But $H$ was an arbitrary element of $\mathfrak{h}$, so $\alpha=0$, which contradicts $\alpha\in \Delta'$. As a result, $X_{\alpha}$ must belong to $\mathfrak{s}$ and therefore $\mathfrak{n}_{\mathfrak{g}}(\mathfrak{s})$ is contained in $\mathfrak{s}$. 
\end{proof}

A subgroup $A$ of $G$ is said to {\it normalize} a subgroup $B$ if $A\leq N_G(B)$. For example, $B$ always normalizes itself.  We have shown in Proposition \ref{LBDS} that bad subgroups normalize a Levi subgroup of a parabolic subgroup in $G$, or a BdS subgroup of $G$.

\begin{cor}\label{selfnorm}
Let $G$ be a reductive $\C$-group and $S$ a connected subgroup containing a Cartan subgroup of $G$, then $N_G(S)/S$ is a finite group.

\end{cor}
\begin{proof}
Lemma \ref{liealgnorm} implies that $N_G(S)/S$ is a discrete group. Since $N_G(S)$ is algebraic, it has a finite number of connected components. Thus, $N_G(S)/S$ is finite. 

\end{proof}

The next lemma, which generalizes~\cite[Theorem 2.9]{FLR} will be used to give a lower bound on the codimension of the bad locus.

\begin{lem}\label{codim} 
Let $G$ be a reductive $\C$-group, $r\geqs 2$ and $H$ be an algebraic subgroup of $G$. Let
$$\varphi_H: G\times \hom(\F_r, H)\to \hom(\F_r,G)$$ be defined by $(g,\rho)\mapsto g\rho g^{-1}$.  Let $\mathcal{H}$ be the image of $\varphi_H$. Then $$\codim_\C(\mathcal{H})\geqs(r-1)\codim_\C(H).$$ 
\end{lem}
\begin{proof} For each $\rho = g\psi g^{-1}\in \mathcal{H}$, with $\psi\co F_r\to H$, the fiber of $\varphi_H$ over $\rho$ contains the subspace $\{(gh, h^{-1} \psi h): h\in H\}$, which is homeomorphic to $H$. Hence each fiber of $\varphi_H$ has dimension at least the dimension of $H$. The result now follows from Hardt's Theorem~\cite{Hardt} (see also~\cite[Corollary 4.2]{Coste}).

\end{proof}

We remark that $\mathcal{H}$ in the above lemma is exactly the set of representations $\rho:\F_r\to G$ which are conjugate in $G$ to a representation with values in $H$.  We now put together the above results to obtain a bound on the codimension of the bad locus.

\begin{thm}\label{codimbad}
Let $G$ be a connected, reductive $\C$-group and $r\geqs 2$. Then 
\[\codim_{\mathbb{C}}(\hom^{bad}(\F_r,G))\geqs 2(r-1)\min\{\mathrm{Rank}(G')\ |\ G' \text{ is a simple factor of } \widetilde{DG}\} .\]
\end{thm}
\begin{proof}
By Proposition \ref{LBDS}, for each bad representation $\rho$ there exists a subgroup $S < G$, which is 
either a Levi subgroup of a proper parabolic subgroup (Type I, Definition \ref{types}) or a BdS subgroup (Type II, Definition \ref{types}), such that the image of $\rho$ 
is contained in the normalizer $N_G(S)$. Recall parabolic subgroups $P$ are in bijection (up to conjugation) with subsets of nodes from the Dynkin diagram of $\mathfrak{g}$, and Levi subgroups of a parabolic are all conjugate by elements of $P$.  So there are finitely many conjugacy classes of Type 1 subgroups. Similarly, by \cite{Tit} there are only finitely many conjugacy classes of Type II. Thus there exists a finite set  $\mathcal{S}$ of subgroups of $G$ such that  the image of each bad representation is contained in a conjugate of $N_G(S)$ for some $S\in \mathcal{S}$. 

There exists a decomposition $\mathfrak{g} = \mathfrak{z}\oplus \mathfrak{g}_1 \oplus \cdots \oplus \mathfrak{g}_n$ with $\mathfrak{z}$ abelian and $\mathfrak{g}_i$ simple for each $i$.
By Lemma \ref{regular2}, we may assume without loss of generality that for each $S\in \mathcal{S}$, its Lie algebra $\mathfrak{s}$ has the form $\mathfrak{z} \oplus \mathfrak{s}_1 \oplus \cdots \oplus \mathfrak{s}_n$, where $\mathfrak{s}_i = \mathfrak{g}_i$ for all but one index $i (S)$, and  $\mathfrak{s}_{i(S)} \subset \mathfrak{g}_{i(S)}$ is either a maximal proper parabolic or a maximal proper BdS subalgebra of $\mathfrak{g}_{i(S)}$.  

Now we apply Lemma \ref{codim}.  Let $\mathcal{N}_S$ be the image of $\varphi_{N_G(S)}$, where $\varphi_{N_G(S)}$ is the map in Lemma \ref{codim}.  So $\hom(\F_r,G)^{bad}\subset \cup_{S\in\mathcal{S}}\mathcal{N}_S.$  Thus, $$\dim_\C(\hom(\F_r,G)^{bad})\leqs \max_{S\in \mathcal{S}}\dim_\C\mathcal{N}_S$$ which implies
$$\codim_\C(\hom(\F_r,G)^{bad})\geqs r\dim_\C(G)-\max_{S\in \mathcal{S}}\dim_\C\mathcal{N}_S,$$ and so $$\codim_\C(\hom(\F_r,G)^{bad})\geqs \min_{S\in \mathcal{S}}\codim_\C\mathcal{N}_S.$$
From Lemma \ref{codim}, for each $S\in \mathcal{S}$, $\codim_\C(\mathcal{N}_S)\geqs(r-1)\codim_\C(N_G(S))$.  Therefore, 
\[\codim_\C(\hom(\F_r,G)^{bad})\geqs (r-1) \min_{S\in \mathcal{S}}\codim_{\mathbb{C}}(N_G(S)).\]

Lemma~\ref{selfnorm} now yields
$$\codim_{\mathbb{C}}(N_G(S))=\codim_{\mathbb{C}}(S),$$
and so $$\codim_\C(\hom(\F_r,G)^{bad})\geqs (r-1) \min_{S\in \mathcal{S}}\codim_{\mathbb{C}}(S).$$ Since $G$ and $S$ are connected it suffices to prove the analogous result for their corresponding Lie algebras. 

In the simple case, the claimed bound is a consequence of the explicit classification of Levi subalgebras of maximal parabolic subalgebras, and maximal BdS subalgebras in simple Lie algebras.  This classification in the case of maximal parabolic subalgebras can be derived from \cite{Borel}.  The classification of BdS subalgebras follows from work of \cite{Tit}.  We tabulate the codimension of all such simple Lie algebras in Tables \ref{Levi} and \ref{BdS} in Appendix \ref{appa}, finding the required codimension bound in each case.

For a general reductive group, from our choice of $\mathcal{S}$, for each $\mathfrak{s}\in \mathcal{S}$, we see that the codimension of $\mathfrak{s}$ is the same as the codimension of $\mathfrak{s}_{i(S)}$ in the simple factor $\mathfrak{g}_{i(S)}$. The desired  codimension bound now follows from the case of simple Lie algebras.
\end{proof}

\begin{rem}
A couple remarks are in order.  When referring to Tables \ref{Levi} and \ref{BdS}, one might be tempted to think there are a few low rank exceptions to the above theorem.  However, this is not the case since there are the following low rank isomorphisms: $\mathfrak{sp}_2(\C)\cong\mathfrak{sl}_2(\C)$, and $\mathfrak{sp}_4(\C)\cong\mathfrak{so}_5(\C)$. And $\mathfrak{so}_4(\C)\cong\mathfrak{sl}_2(\C)\times \mathfrak{sl}_2(\C)$ and so is not simple. The last case that is not addressed in the tables is when $G=\SO_2(\C)$.  In that case, it is easy to see that the bad locus is empty.  So $\codim_\C\hom(\F_r,G)^{bad}=\dim_\C G^r=r\geqs 2\mathrm{Rank}(DG)=0$ since the derived subgroup of an abelian group is trivial, and so has rank $0$.
\end{rem}

\section{The Singular Locus}\label{singsec}

In this section we undertake a close analysis of singularities in $\X_r (G)$ in both the algebraic and topological categories.

\subsection{Local Structure}

For a subset $X$ of a group $G$, let $Z_G(X)$ be the group of elements in $G$ that commute with all elements in $X$.  We let $Z_G(\rho)$ for $\rho\in \hom(\F_r,G)$ be $Z_G(\rho(\F_r))$.

\begin{lem}
Let $G$ be a reductive $\C$-group.  Let $\Gamma$ be a subgroup of $G$ containing $Z(G)$ such that $\Gamma/Z(G)$ is finite, and non-abelian.  If $\gamma Z(G)\in \Gamma/Z(G)$ is not central in $\Gamma/Z(G)$, then $\dim_\C Z_G(\Gamma)<\dim_\C Z_G(\gamma)$.  
\end{lem}

\begin{proof}
Since $\Gamma/Z(G)$ is finite, $\gamma$ needs to be semisimple (compare Proposition \ref{commsemi}). Write $\gamma=\exp(X)$ with $X\in \mathfrak{g}$. Then $Z_G(\gamma)$ contains the $1$-parameter subgroup generated by $X$ while $Z_G(\Gamma)$ does not, whence the result. 
\end{proof}

\begin{prop}\label{ab-gen}
Let $G$ be a reductive $\C$-group and $r\geqs 2$. Bad representations $\F_r\to G$ with abelian $($non-trivial$)$ stabilizers in $PG$ are Euclidean dense $($and hence Zariski dense$)$ in the locus of bad representations $\F_r\to G$.
\end{prop}
\begin{proof}
Let $\rho$ be a bad representation and assume its stabilizer $Z_G(\rho)/Z(G)$  in $PG$ is non-abelian. Let $\xi$ be in $Z_G(\rho)$ and assume $\xi Z(G)$ is not central in $Z_G(\rho)/Z(G)$. 

By the preceding lemma applied to $Z_G(\rho)$ and $\xi$, $\dim_\C Z_G(Z_G(\rho))<\dim_\C Z_G(\xi)$. It follows that $\dim_\C \hom(\F_r,Z_G(Z_G(\rho)))<\dim_\C \hom(\F_r,Z_G(\xi))$, and therefore the complement of $\hom(\F_r,Z_G(Z_G(\rho)))$ is dense in $\hom(\F_r,Z_G(\xi))$. 

Given an open subset of $\hom(\F_r,Z_G(\xi))$ containing $\rho$, up to intersecting with the irreducible locus, we may assume that it is an open subset of the irreducible locus.  Call this neighborhood $U$. From the previous paragraph, there exist representations in $Z_G(\xi)$ that do not commute with $Z_G(\rho)$. Therefore, in $U$ there is an irreducible representation $\rho'$ commuting with $\xi$ but not with $Z_G(\rho)$ and thus a bad representation $\rho'$ with a strictly smaller centralizer. 

If $Z_G(\rho')/Z(G)$ is abelian, we are done and if it is not we may repeat the above argument to find a bad representation with a strictly smaller stabilizer in $PG$. Since the stabilizer is finite, this process will eventually stop. 
\end{proof}

For a group $G$ acting on a space $X$, we denote the fixed locus by $X^G$. We now show that for bad representations, the action of the stabilizer on cohomology never includes pseudoreflections, that is, finite order elements whose fixed locus is of codimension $1$ (compare \cite[Lemma 8.5]{Ri}).

\begin{prop}\label{nopseudo}
Let $G$ be a connected, reductive $\C$-group and $r\geqs 2$.  If $\rho:\F_r\to G$ is a bad representation commuting with $\xi\notin Z(G)$. Then \[\codim_\C H^1(\F_r;\mathfrak{g}_{\Ad_\rho})^{\langle \xi\rangle}\geqs 2.\]
\end{prop}

\begin{proof}
First, $\Ad_\xi$ is of finite order since $Z_G(\rho)/Z(G)$ is finite. Then we have the decomposition:
\[\mathfrak{g}=\mathfrak{g}^{\langle \xi\rangle}\oplus V,\]
where $V$ is a sum of eigenspaces of $\xi$ for eigenvalues $\neq 1$. It follows that the decomposition above is stable under $\rho$. As a result, we have :
\[ H^1(\F_r;\mathfrak{g}_{\Ad_\rho})= H^1(\F_r;\mathfrak{g}_{\Ad_\rho}^{\langle \xi\rangle})\oplus  H^1(\F_r;V)\]
where the action of $\xi$ is determined by the action on the coefficients of cohomology.  In particular, \[H^1(\F_r;\mathfrak{g}_{\Ad_\rho})^{\langle \xi\rangle}=H^1(\F_r;\mathfrak{g}_{\Ad_\rho}^{\langle \xi\rangle}).\]
The dimension of $H^1(\F_r;V)$ is computed as: $$r\dim_\C V-(\dim_\C V-\dim_\C V^{\F_r})=(r-1)\dim_\C V+\dim_\C V^{\F_r}.$$

 It follows that 
\[\codim_\C H^1(\F_r;\mathfrak{g}_{\Ad_\rho})^{\langle \xi\rangle}\geqs (r-1)\dim_\C V\geq \dim_\C V.\]
Since $V$ is, by definition, complementary to $\mathfrak{g}^{\langle \xi\rangle}=\mathfrak{z}(\xi)=\mathrm{Lie}(Z_G(\xi))$, we have $\dim_\C V=\codim_\C Z_G(\xi)$ and since $Z_G(\xi)$ is regular and reductive (by Lemma~\ref{semisimplecentralizers}), $\codim_\C Z_G(\xi)$ is even and therefore greater than or equal to $2$ (since $\xi\notin Z(G)$). 
  \end{proof}

\begin{lem}\label{sing-lem}
A representation
$[\rho]\in \X_r(G)$ is singular $($respectively ugly$)$ if and only if $[0]$ is an algebraic $($respectively topological$)$ singularity in $H^1(\F_r;\mathfrak{g}_{\Ad_\rho})\aq \mathrm{Stab}(\rho)$.
In fact, if $[\rho]$ is not ugly in $\X_r (G)$ and $ \mathrm{Stab}(\rho)$ is finite, then $H^1(\F_r;\mathfrak{g}_{\Ad_\rho})\aq \mathrm{Stab}(\rho)$ is homeomorphic to a Euclidean space.
\end{lem}

\begin{proof}
Let $[\rho]\in \X_r(G)$. Without loss of generality, we can assume $\rho$ is polystable since there exists a unique closed orbit in $\pi_G^{-1}([\rho])$. Using the Luna Slice Theorem \cite{DrJ,Lu1}, \cite{FL2,Si4,HP} shows for every polystable $\rho\in \hom(\F_r,G)$, there is a local model of $[\rho]\in\X_r(G)$ of the form $H^1(\F_r;\mathfrak{g}_{\Ad_\rho})\aq \mathrm{Stab}(\rho)$, where $H^1(\F_r;\mathfrak{g}_{\Ad_\rho})$ is group cohomology with coefficients in the $\Ad_\rho$-module $\mathfrak{g}$, and $\mathrm{Stab}(\rho)$ is the stabilizer of $\rho$ in $PG$.  The $\C$-points of $H^1(\F_r;\mathfrak{g}_{\Ad_\rho})$ are isomorphic to a $\C$-vector space $V$ where $0$ corresponds to $[\rho]$, and $\Gamma:=\mathrm{Stab}(\rho)$ is a subgroup of $PG$ acting linearly (and so fixes $0$). The first part then follows by Proposition \ref{sing-princ}.  

The final statement of the lemma follows from the more general fact that if $\Gamma$ is a finite group acting linearly on a $\C$-vector space $V$, and there is a neighborhood of $[0] \in V\aq \Gamma$ that is homeomorphic to Euclidean space, then $V\aq \Gamma$ itself is homeomorphic to Euclidean space. To prove this, note that if $\langle \,,\,\rangle$ is any Hermitian inner product on $V$, then $(v, w):= \sum_{g\in G} \langle v, w \rangle$  is a $\Gamma$--invariant Hermitian inner product, whose associated norm $|v| := \sqrt{(v,v)}$ is also $\Gamma$--invariant. Since all norms on finite dimensional $\C$-vector spaces are linearly equivalent, balls and spheres in this $\Gamma$-invariant norm are homeomorphic to ordinary balls and spheres.  Now $\Gamma$ acts on each sphere $$S_\epsilon = \{v\in V :\, |v| = \epsilon\},$$ and in fact by linearity of the action, the scaling map $s:S_\epsilon \to S_{\epsilon'},$ given by scaling each vector to have the desired length, is $\Gamma$-equivariant (that is, $g(sv) = sg(v)$, since $s$ is just scalar multiplication by $\epsilon'/\epsilon$).  We now claim that $V\aq\Gamma$ is homeomorphic to the open cone $$\mathcal{C}_\R(S_1\aq\Gamma):= (S_1\aq\Gamma \cross [0, \infty)) / (S_1\aq \Gamma \cross \{0\}),$$ where $S_1$ is the unit sphere in $V$ (with respect to the $\Gamma$-invariant metric). Indeed, the map $S_1 \cross [0, \infty) \to V$ given by sending $(x, t)\mapsto tx$ is $\Gamma$-equivariant and is an open map, and hence descends to an open map $S_1\aq\Gamma \cross [0, \infty) \to V\aq \Gamma$.  This map is surjective and injective away from $0$, and becomes bijective (hence a homeomorphism by openness) once we factor it through the open cone.  Now Kwun's theorem \cite{Kw} on uniqueness of open cone neighborhoods says that if there exist a neighborhood $U$ of $[0]$ in $V\aq\Gamma$ that is homeomorphic to $\C^n \cong\mathcal{C} (S^{2n-1})$, then since $U$ and $V\aq\Gamma$ are both open cone neighborhoods of $[0]$, they must be homeomorphic.
\end{proof}

In the next subsection we will need to closely analyze the local model 
$$H^1(\F_r;\mathfrak{g}_{\Ad_\rho})\aq \mathrm{Stab}(\rho)$$ 
for reducible representations $\rho$. We will show (in the proof of Theorem~\ref{redugly-thm}) that generically (in the reducible locus) $\mathrm{Stab}(\rho)$ is isomorphic to $\mathbb{C}^*$, and there is a decomposition
 $$H^1(\F_r;\mathfrak{g}_{\Ad_\rho}) \cong \bigoplus_{n=-N}^{N}\mathbb{C}^{d_n},$$
for some non-negative integers $d_{-N},\dots,d_N$ satisfying $d_n=d_{-n}$, so that  the action of $\mathbb{C}^*\isom \mathrm{Stab}(\rho)$ is the diagonal action induced by $\lambda\cdot v=\lambda^n v$ for $v\in \mathbb{C}^{d_n}$.
We will call the numbers $n$ so $d_n\not=0$ weights, and we call the action of $\C^*$ on $\C^{d_n}$ the scalar action of weight $n$. Let $L:=\oplus_{n=1}^N\C^{d_{-n}}$ and $R:=\oplus_{n=1}^N\C^{d_n}$.  The diagonal actions of $\C^*$ on $L$ and $R$ give rise to a $\C^*\times \C^*$ action on $L\oplus R$ such that the action of the diagonal $\C^*\cong \Delta\subset \C^*\times \C^*$ recovers  the action of $\mathrm{Stab}(\rho)$.  Let $\pi_\Delta:L\oplus R\to (L\oplus R)\aq \Delta$ be the GIT projection map with respect to the action by $\Delta$.  Notice that the $\Delta$-orbits in $L\oplus R$ are all closed outside of $\pi_\Delta^{-1}([0])=L\oplus\{0\}\cup \{0\}\oplus R$.  Let $(L\oplus R)^*:=(L\oplus R)-\pi_\Delta^{-1}([0])$ and note that $(L\oplus R)^*\aq \Delta = (L\oplus R)^*/\Delta \cong(L\oplus R)\aq \Delta-\{[0]\}$.  Then $\pi_\Delta:(L\oplus R)^*\to (L\oplus R)^*/ \Delta$ is a geometric quotient.  Note also that $(L\oplus R)^*/ (\C^*\times \C^*)\cong \mathbb{P}(L)_w\times \mathbb{P}(R)_w$ where $\mathbb{P}(L)_w\cong \mathbb{P}(R)_w$ are {\it weighted} projective spaces (see \cite{Dol-wpv}).

To understand this particular local model, we need the following lemmata.

\begin{lem}\label{weighted-action-lem}
Let $k_1, \ldots k_m$ be non-zero integers. Define an action of $\mathbb{T} = \mathbb{C}^*$ on $\mathbb{C}^{m}$ by $t\cdot (v_1,\ldots, v_m)=(t^{k_1} v_1, \ldots, t^{k_m} v_m)$. Let 
$\Upsilon =  \prod_{i=1}^m \mu_i$, where $\mu_i \subset \C^*$ is the group of $k_i$-th roots of unity, and let
$\Gamma = \C^* \cross \prod_{i=1}^m \mu_i$. Define an action of $\Gamma$ on $\mathbb{C}^{m}$ by $(\lambda, \zeta_1, \ldots, \zeta_m) \cdot (v_1, \ldots, v_m) = (\lambda^{\epsilon_1} \zeta_1 v_1, \ldots, \lambda^{\epsilon_m} \zeta_m v_m)$, where $\epsilon_i = k_i/|k_i|$. Then the map $f\co \C^m \to \C^m$, given by $$(v_1,\ldots, v_m)\mapsto (v_1^{|k_1|},\ldots, v_m^{|k_m|}),$$ induces a homeomorphism $\C^m/\Gamma \srm{\isom} \C^m/\mathbb{T}.$

Consequently, in the notation established above,
\begin{equation}\label{eq:M}(L\oplus R)^*/\Delta \isom ((L\oplus R)^*/\C^*)/\Upsilon,\end{equation}
where $\C^*$ acts diagonally by scalar multiplication $($weight $1$ on $R$ and weight $-1$ on $L)$ and $\Upsilon$ is a finite group acting trivially on the homology of $(L\oplus R)^*/\C^*$.
\end{lem}

\begin{proof} It follows from the definitions that the composition $\C^m \srm{f} \C^m \to \C^m/\mathbb{T}$ is invariant under the action of $\Gamma$ in the domain. A short computation shows that the induced map $\C^m/\Gamma \srm{\isom} \C^m/\mathbb{T}$ is bijective, and it is an open map because $f$ and the quotient map $\C^m\to \C^m/\mathbb{T}$ are both open. The homeomorphism (\ref{eq:M}) is obtained by restriction, and the action of $\Upsilon$ is homologically trivial because it extends to the ordinary scalar action of the path-connected group $(\C^*)^{m}$, where $m = \dim_\C (L\oplus R)$.
\end{proof}

\begin{lem}\label{hom-loc-model-lem}
The space $(L\oplus R)^*/\Delta$ has non-trivial rational homology in dimensions
$0, 2, \ldots, 2M, 2M+1, \ldots, 4M+1$, where $M = \dim_{\C} (L) - 1 =  \dim_{\C} (R) - 1$ is the $($complex$)$ dimension of the $($unweighted$)$ projective spaces $\mathbb{P} (L) \isom \mathbb{P} (R)$. In particular, $(L\oplus R)^*/\Delta$ does  \e{not} have the rational homology of a sphere unless $L$ and $R$ are one-dimensional.
\end{lem}
\begin{proof} All homology groups will be taken with rational coefficients.
Lemma~\ref{weighted-action-lem} shows that
$$H_* ((L\oplus R)^*/\Delta) \isom H_* ((L\oplus R)^*/\C^*),$$
where on the right-hand side, $\C^*$ is acting by scalar multiplication in the vector space $L\oplus R$, with weight $1$ on $R$ and weight $-1$ on $L$. This action is the diagonal of the $(\C^* \cross \C^*)$-action given by scalar multiplication in $L$ and $R$ separately (again with weight $1$ on $R$ and weight $-1$ on $L$). Consider the quotient map
$$(L\oplus R)^*/\C^* \maps (L\oplus R)^*/(\C^*\cross \C^*) = \mathbb{P} (L) \cross \mathbb{P} (R).$$
This map is the quotient map for the natural action of the quotient group $(\C^* \cross \C^*)/\C^* \isom \C^*$, and in particular is a locally trivial $\C^*$--bundle. Since $\C^* \heq S^1$ has homology in dimensions $0$ and $1$ only, the Serre spectral sequence for this bundle has all differentials equal to zero after the second page. The K\"unneth Theorem shows that the homology of $\mathbb{P} (L) \cross \mathbb{P} (R)$ is concentrated in even degrees, and
$$H_{2k}(\mathbb{P} (L) \cross \mathbb{P} (R)) \isom  \mathbb{Q}^{k+1}$$
for $0\leqs k\leqs M$ and 
$$H_{2M+2k}(\mathbb{P} (L) \cross \mathbb{P} (R)) \isom  \mathbb{Q}^{M -k+1}$$
for $0 < k \leqs M$.
This implies that every differential out of the groups 
$$E^2_{2k, 0} \isom H_{2k} (\mathbb{P} (L) \cross \mathbb{P} (R))$$ 
with $0\leqs k \leqs M$ has non-zero kernel, while every differential into the groups 
$$E^2_{2M + 2k, 1} \isom H_{2M+2k} (\mathbb{P} (L) \cross \mathbb{P} (R))$$
with $0  \leqs k \leqs M$ has non-trivial cokernel. Hence $E^\infty_{2k, 0}$ is non-zero for 
$0\leqs k \leqs M$, while $E^\infty_{2M + 2k, 1}$ is non-zero for $0  \leqs k \leqs M$, giving the desired result.
\end{proof}

\begin{lem}\label{homology-lem}
Let $X$ be a space with the rational homology of a point, and consider a $($closed$)$ point $x_0\in X$ such that $H_* (X-\{x_0\}; \mathbb{Q})$ is finitely generated in each degree. If $X\cross \bbR^n - \{(x_0, 0)\}$ has the same rational homology as the sphere $S^k$ for some $k\geqs n-1$, then $X-\{x_0\}$ has the same rational homology as $S^{k-n}$ $($where, by convention, we take $S^{m}$ to be the empty space if $m<0$$)$, and $X\cross \bbR^n - \{(x_0, 0)\}$ cannot have the same rational homology as $S^k$ if $0 \leqs k < n-1$.
\end{lem}
\begin{proof} We prove the statement by induction on $n$. Again, all homology groups are rational.

For $n=0$ the statement is immediate. Next, consider the case $n=1$, and say $X\cross \bbR - \{x_0, 0\}$ has the rational homology of $S^k$ for some $k\geqs 1$ (we will consider the case $k=0$ separately).  Consider the Mayer--Vietoris sequence associated to the decomposition 
$$X\cross \bbR - \{(x_0, 0)\} = (X\cross (\bbR - \{0\})\cup ((X-\{x_0\}) \cross \bbR).$$ 
Note the homotopy equivalences $X\cross (\bbR - \{0\}) \heq X\coprod X$ and $(X-\{x_0\}) \cross \bbR\heq X-\{x_0\}$. These open sets intersect in 
$$(X-\{x_0\}) \cross (\bbR - \{0\}) \heq (X-\{x_0\})\coprod (X-\{x_0\}).$$
The Mayer-Vietoris sequence now has the form (in part)
\begin{eqnarray*}
0\maps H_k (X - \{x_0\}) \oplus  H_k (X - \{x_0\}) \srm{i} H_k (X-\{x_0\}) \maps H_{k} (S^k) = \mathbb{Q}\hspace{.3in}\\
\maps H_{k-1} (X - \{x_0\}) \oplus  H_{k-1} (X - \{x_0\}) \maps H_{k-1} (X-\{x_0\}) \maps 0.
\end{eqnarray*}
 The map $i$ must be injective, which implies that $ H_k (X - \{x_0\}) = 0$ (here we use the assumption that $H_k (X - \{x_0\})$ is finitely generated). The remaining short exact sequence 
$$0\maps  \mathbb{Q}
\maps H_{k-1} (X - \{x_0\}) \oplus  H_{k-1} (X - \{x_0\}) \maps H_{k-1} (X-\{x_0\}) \maps 0$$
 shows that $H_{k-1} (X - \{x_0\}) \isom \mathbb{Q}$, as desired. Similar (in fact, simpler) reasoning gives $\wt{H}_{*} (X - \{x_0\}) = 0$ for $*\neq k-1$, as desired. Now say $k=0$, so that $X\cross \mathbb{R} - \{(x_0, 0)\}$ has the rational homology of $S^0$. 
We claim that in this case we must have  $X = \{x_0\}$, in which case the statements of the lemma hold trivially.  Indeed, when $k =0$, the above Mayer--Vietoris sequence has the form
 \begin{eqnarray*}
  H_0 (X - \{x_0\}) \oplus  H_0 (X - \{x_0\}) \injects H_0 (X-\{x_0\}) \oplus H_0 (X) \oplus H_0 (X) \surjects H_{0} (S^0),
\end{eqnarray*}
and since $H_0 (X) = \mathbb{Q}$, we find that $H_0 (X - \{x_0\}) = \{0\}$, meaning that $X - \{x_0\}$ is empty.
 This proves the lemma for $n=1$ (note that in this case the last statement of the lemma is vacuous).
  
Now we assume the statements of the lemma for $n=0, 1, \ldots, m-1$ and consider the case for $n=m$ ($m\geqs 2$). 

Say $X$ has the rational homology of a point, and assume $X\cross \bbR^m - \{(x_0, 0)\}$  has the same rational homology as $S^k$ for some $k\geqs m-1$.
We have 
$$X\cross \bbR^m - \{(x_0, 0)\} = (X \cross \bbR^{m-1})\cross \bbR - \{(x_0, 0, 0)\}.$$ 
Note that $Y:=X\cross \bbR^{m-1} \heq X$ has the rational homology of a point, and $Y - \{(x_0, 0)\}$ has finitely generated rational homology by a Mayer--Vietoris argument similar to that above. The induction hypothesis for $n=1$ implies that $X\cross \bbR^{m-1} - \{(x_0, 0)\}$ has the same rational homology as $S^{k-1}$, and the induction hypothesis for $n=m-1$ shows that $X-\{x_0\}$ has the same rational homology as $S^{k-1-(m-1)} = S^{k-m}$, as desired.
We also need to show that $X\cross \bbR^m - \{(x_0, 0)\}$ cannot have the homology of $S^k$ with $0 \leqs k < m-1$. If it did, then writing 
$$X\cross \bbR^m - \{(x_0, 0)\} = (X\cross \bbR) \cross \bbR^{m-1} - \{(x_0, 0, 0)\},$$
the induction hypothesis for $n=m-1$ gives a contradiction if $0\leqs k < m-2$, and it remains to show that  
$X\cross \bbR^m - \{(x_0, 0)\}$ cannot have the homology of $S^{m-2}$. If it did, then writing 
$$X\cross \bbR^m - \{(x_0, 0)\} = (X\cross \bbR^{m-1}) \cross \bbR - \{(x_0, 0, 0)\},$$
then the induction hypothesis for $n=1$ shows that $X\cross \bbR^{m-1} - \{(x_0, 0)\}$ has the homology of $S^{m-3}$, and repeating this argument we eventually find that 
$X\cross \bbR -  \{(x_0, 0)\}$ has the homology of the empty space $S^{-1}$, but this is impossible as $X\cross \bbR -  \{(x_0, 0)\}$ is non-empty.
\end{proof}

\begin{lem}\label{weighted-lem}
Let $d_{-N},\dots,d_N$ be non-negative integers such that $d_n=d_{-n}$. Define the action of $\mathbb{C}^*$ on $\mathbb{C}^{d_n}$ by $\lambda\cdot v=\lambda^n v$. Then 
\[  \left(\bigoplus_{n=-N}^{N}\mathbb{C}^{d_n}\right)\aq\mathbb{C}^*\isom \mathbb{C}^{d_0} \cross 
\left(\bigoplus_{n\neq 0}\mathbb{C}^{d_n}\right)\aq\mathbb{C}^*
\]
has a topological singularity at $[0]$ if and only if $\sum_{n\geqs 1} d_n>1$. 
\end{lem}

\begin{proof} First, say $\sum_{n\geqs 1} d_n>1$. 
Let $V = \bigoplus_{n=-N}^{N}\mathbb{C}^{d_n}$ and let $X = V\aq \C^*$. The standard contracting homotopy of $V$, namely $H_t (v) = tv$, induces a contracting homotopy of $X$, so in particular $X$ has the homology of a point.
Moreover, Lemma \ref{hom-loc-model-lem} shows that the rational homology of $\left(\bigoplus_{n\neq 0}\mathbb{C}^{d_n}\right)\aq\mathbb{C}^*$ is finitely generated, and is not that of a sphere.
By Lemma \ref{homology-lem}, the homology of $X-\{[0]\}$ is not that of a sphere either. However, if $X$ had a Euclidean neighborhood $U \isom \mathbb{R}^k$ around $[0]$, we would have 
$$H_* (X, X - \{[0]\}) \isom H_* (\mathbb{R}^k, \mathbb{R}^k - \{0\})$$
by excision, and by comparing the long exact sequences of these pairs we find that $H_* (X - \{[0]\}) \isom H_* ( \mathbb{R}^k - \{0\}) \isom H_* (S^{k-1})$. Hence $X$ has a topological singularity at $[0]$.

Conversely, if $\sum_{n\geqs 1} d_n=1$, then $(\C\oplus\C^{d_0}\oplus \C)\aq \C^*\cong \C^{d_0}
\cross (\C \oplus \C)\aq\C^*$, where on the right $\C^*$ acts with weights $-n$ and $n$ on the left and right factors (respectively). The multiplication map $\C\cross \C\to \C$ induces a homeomorphism
$ (\C \oplus \C)\aq\C^* \srm{\isom} \C$, so we find that $(\C\oplus\C^{d_0}\oplus \C)\aq \C^*$ is homeomorphic to $\C^{d_0+1}$ in this case.
\end{proof}

\subsection{Algebraic and Topological Singularities}

Richardson showed the singular locus of $\X_r(G)$ is precisely the union of the bad locus with the reducible locus if $G$ is semisimple, $r\geqs 2$ and the Lie algebra of $G$ does not have any rank 1 simple factors \cite{Ri}.  \cite[Theorem 7.4]{FLR} extends Richardson's theorem to connected reductive groups $G$ such that the simple factors of the Lie algebra of its derived subgroup $DG$ have rank 2 or more (addressing Conjectures 3.34 and 4.8 in \cite{FL2}).

In this subsection, we generalize these results by showing that if $r\geqs 3$ and $G$ is connected and reductive (but allowing its derived subgroup to have local rank 1 factors), then the singular locus coincides with the union of the bad locus with the reducible locus.  We also generalize results in \cite{FL2}, by showing all algebraic singularities in $\X_r(G)$ are in fact topological (ugly); which shows that $\X_r(G)$, as a class of varieties, have properties in common with normal surfaces (\cite{Mumford}).

\begin{rem}\label{rankexamples}
\cite[Remark 3.33]{FL2} shows that $\XC{2}(\p\SL_2(\C))$ has smooth points which are reducible and singular points which are irreducible; so a condition on the rank of $G$ when $r=2$ is necessary. \cite[Examples 7.2 and 7.3]{FLR} show that there are Lie groups $H$ of arbitrarily large rank with the property that $\XC{2}(H)$ has smooth reducibles and singular irreducibles (and $H$ does not have to be a product with a rank 1 Lie group for this to happen; although $H$ does need to be a local product with a rank 1 Lie group). 
\end{rem}

\begin{thm}\label{conj-thm}Let $G$ be a connected, reductive $\C$-group.  Assume either $r\geqs 2$ and the Lie algebra of $DG$ has no rank 1 simple factors, or that $r\geqs 3$ with no additional conditions on $G$. Then $$\XC{r}(G)^{sing}=\XC{r}(G)^{red}\cup\X_r(G)^{bad},$$ and all points in $\XC{r}(G)^{bad}$ are orbifold singularities.
\end{thm}

\begin{proof}
We prove the theorem in the following steps:
\begin{enumerate}
 \item[$(1)$] $\XC{r}(G)^{red}\subset \XC{r}(G)^{sing}$,
 \item[$(2)$] $\X_r(G)^{bad}\subset \XC{r}(G)^{sing}$, and all points in $\XC{r}(G)^{bad}$ are orbifold singularities,
 \item[$(3)$] $\XC{r}(G)^{good}=\mathcal{X}_r (G)$.
\end{enumerate}

If $r\geqs 2$ and the Lie algebra of $DG$ has only simple factors of rank 2 or more, then this is the content of Richardson \cite{Ri} in the case with $G$ is semisimple, and \cite[Theorem 7.4]{FLR} in the case when $G$ is connected and reductive.  So we now assume $r\geqs 3$.

The second part of (2) follows from the first part of (2) since $\X_r(G)^{irr}$ is always an orbifold and $\X_r(G)^{good} \subset \mathcal{X}_r (G)$ is always a smooth manifold.  This also shows that (3) follows from (1) and (2). 

(1): Let $G$ be a connected, reductive $\C$-group, $DG$ the derived subgroup $[G,G]$, and $\widetilde{DG}$ the universal cover of $DG$.  Then $\widetilde{DG}=\prod_i G_i$ is a finite product of simple Lie groups and \cite[Proposition 2.9]{FL4} says $\X_r(\prod_i G_i)\cong \prod_i\X_r(G_i)$. Now, for $r\geqs 3$ and $G$ is simple, we have $\X_r(G)^{red}\subset \X_r(G)^{sing}$: for $\mathrm{Rank}(G)\geq 2$ this follows by \cite[Theorem 7.4]{FLR}, and for $\mathrm{Rank}(G) = 1$ it follows by\cite[Theorem 3.21]{FL2} and \cite[Corollary 7.9]{FLR}.  Observe that $[\rho]\in\X_r(\prod_i G_i)\cong \prod_i\X_r(G_i)$ is reducible if and only if it is reducible in some $\X_r(G_i)$, and likewise it is singular if and only if it is singular in some $\X_r(G_i)$. 

Thus,  if $r\geqs 3$, we have $\X_r(\widetilde{DG})^{red}\subset \X_r(\widetilde{DG})^{sing}$. However, \cite[Theorem 7.8]{FLR} proves that if $\X_r(\widetilde{DG})^{red}\subset \X_r(\widetilde{DG})^{sing}$, then $\X_r(G)^{red} \subset \X_r(G)^{sing}$, proving item (1).

(2): We will use the following lemma from \cite{NR}:

\begin{lem}{\cite[Lemma 4.4]{NR}}\label{NR-lem} Let $f:X\to Y$ be a morphism from an $n$-dimensional complex manifold onto a normal $n$-dimensional variety $Y$.  If the set $S$ of points at which $f$ is not locally injective is of codimension at least 2, then $f(S)$ is the singular set of $Y$. 
\end{lem}

Let $[\rho]\in \X_r(G)$ be bad. Then $\rho$ is polystable since $\rho$ is irreducible and all irreducible representations are stable. By Lemma \ref{sing-lem}, there is a local model of $[\rho]\in\X_r(G)$ of the form $V\aq\Gamma:=H^1(\F_r;\mathfrak{g}_{\Ad_\rho})\aq \mathrm{Stab}(\rho)$, where $\Gamma:=\mathrm{Stab}(\rho)$ is the finite stabilizer of $\rho$ in $PG$.  We note that $V$ has the same dimension as $\X_r(G)$ since $\rho$ is irreducible, and since $\Gamma$ is finite the dimension of $V\aq\Gamma$ is also equal to $\X_r(G)$.

Since $\Gamma$ is finite, every point in $V$ where $\pi_\Gamma:V\to V\aq \Gamma$ is not locally injective has a non-trivial stabilizer (in fact, the reverse implication also holds) and so Proposition \ref{nopseudo} allows us to apply Lemma \ref{NR-lem}. We thus conclude that $[0]$ is singular in $V\aq \Gamma$ and therefore $[\rho]$ is singular in $\X_r(G)$ by Lemma \ref{sing-lem}. Thus, the first part of item (2) holds, and as noted above the theorem follows. 
\end{proof}

\begin{rem}
We note that item $(1)$ in the above proof  resolves the first part of \cite[Conjecture 3.34]{FL2}.  We note (again) that \cite[Examples 7.2 and 7.3]{FLR} show when $r=2$,  item $(1)$ may be false if $DG$ locally has rank 1 factors, so item $(1)$ in Theorem \ref{conj-thm} is sharp and cannot be generally improved.
\end{rem}

As shown in \cite{FL2}, if $G=\SL_n(\C)$ or $\GL_n(\C)$ then generally reducible representations are ugly.  We now generalize this result.

\begin{thm}\label{redugly-thm}
All reducible representations are ugly if $r\geqs 3$ for any connected, reductive $\C$-groups $G$, or if $r\geqs 2$ and the Lie algebra of $DG$ has only simple factors of rank 2 or more.
\end{thm}

\begin{proof}
First note that if a point $[\rho]\in\X_r(G)$ is not ugly, then there exists a Euclidean open set around $[\rho]$ that is homeomorphic to a Euclidean ball.  Thus, all points in that Euclidean set are also not ugly.  And so the collection of non-ugly points in $\X_r(G)$ is an open set; that is, being ugly is a closed condition in $\X_r(G)$.  Since the conjugate of an ugly representation is still ugly (as it gives the same point in $\X_r(G)$), we conclude that the ugly locus in $\hom(\F_r,G)$ is also closed.

This alone shows that all reducibles are ugly in $\X_r(\SL_n(\C))$ if $r\geqs 3$ and $n\geqs 2$ or if $n\geqs 3$ and $r\geqs 2$, since \cite{FL2} shows that a Euclidean dense set of reducibles is ugly in these cases.  We now generalize this to arbitrary $G$.

Let $G$ be a connected, reductive $\C$-group. We say that $L$ is {\it quasi-irreducible} if it is a Levi subgroup of a maximal parabolic subgroup of $G$. Let $\rho :\F_r \to G$ be a representation. We say that $\rho$ is {\it quasi-irreducible} if its Zariski-closure is quasi-irreducible.
(This definition can be viewed as a generalization of the reducible representations shown to be ugly in \cite{FL2}; namely the representations in $\SL_n(\mathbb{C})$ conjugate to a representation having two non-trivial irreducible blocks.)

Up to conjugation we have only finitely many maximally non-irreducible subgroups in $G$ (at most the number of conjugacy classes of maximal parabolic subgroups of $G$; that is, the rank of $G$). 

We claim that quasi-irreducible representations are generic.  Precisely, the subset of quasi-irreducible representations $\hom^{qi}(\F_r,G)$ is Euclidean dense in the (constructible) subset of polystable, reducible representations in $\hom (\F_r, G)$.
This simply comes from the fact that completely reducible and non-irreducible representations $\rho:\F_r \to G$ have to be contained  in a parabolic subgroup of $G$ and therefore a maximal parabolic subgroup. Since they  are also completely reducible they factor through the inclusion of a maximally non-irreducible subgroup of $G$. Now the result follows from the general fact that Zariski-dense representations of a free group of rank at least 2 in a connected reductive group are Euclidean dense in the representation variety of that reductive group by Proposition \ref{zd-prop}.

Now, we prove that quasi-irreducible representations are ugly. First let us begin with the Zariski-closure of $\rho(\F_r)$, which we denote by $L$. Let us show that $Z_G(L)/Z(G)$ is isomorphic to $\mathbb{C}^*$. 

Fix a Cartan subgroup $H$ of $G$. Up to conjugation we may assume that $H$ is a subgroup of  a Levi subgroup $L$   of the maximal parabolic subgroup $P$. Since $H$ contains the center of $G$, we may write its Lie algebra as $\mathfrak{z}\oplus \mathfrak{h}$, where $\mathfrak{z}$ is the center of $\mathfrak{g}$.
Let $\mathfrak{l} \subset \mathfrak{p} \subset \mathfrak{g}$ denote the Lie algebras of $L$ and $P$.

Recall that we have the following decomposition:
\[\mathfrak{g}=\mathfrak{z}\oplus \mathfrak{h}\oplus \bigoplus_{\alpha\in \Delta}\mathfrak{g}_{\alpha},\]
where as usual, $\Delta$ is the set of generalized eigenvalues of $\mathfrak{h}$ (these are linear forms on $\mathfrak{h}$) and $\mathfrak{g}_{\alpha}$ are the corresponding generalized eigenspaces of dimension $1$. We fix a non-zero vector $X_{\alpha}$ in each. We recall that we may choose in $\Delta$ a set of simple roots $\alpha_1,\dots, \alpha_m$ so that any root is a positive sum of these $\alpha_i$ or a negative sum of these $\alpha_i$. Define $(H_1,\dots,H_m)$ to be the dual basis in $\mathfrak{h}$ to $(\alpha_1,\dots,\alpha_m)$ in  $\mathfrak{h}^*$. 

For a fixed $\alpha_i$, we may construct:
\[\mathfrak{p}_i:=\mathfrak{z}\oplus\mathfrak{h}\oplus \bigoplus_{\alpha\in \Delta^+}\mathfrak{g}_{\alpha}\oplus \bigoplus_{\alpha\in \Delta^-\cap\langle \alpha_j\mid j\neq i\rangle}\mathfrak{g}_{\alpha}.\]
Up to conjugation, the Lie algebras of maximal parabolic subgroups are of the form above. Thus we may assume that $\mathfrak{p}=\mathfrak{p}_i$. Finally, if we define 
\[\mathfrak{l}_i:=\mathfrak{z}\oplus\mathfrak{h}\oplus \bigoplus_{\alpha\in \Delta\cap\langle \alpha_j\mid j\neq i\rangle}\mathfrak{g}_{\alpha},\]
then $\mathfrak{l}_i$ is conjugate to the Lie algebra of each Levi subgroup of  $P$. Thus we may assume that $\mathfrak{l}=\mathfrak{l}_i$. 

Now note that $P=\mathrm{Rad}_U(P)\rtimes L$ and since both $P$ and its unipotent radical $\mathrm{Rad}_U(P)$ need to be connected, $L$ is connected as well. It follows that for $x\in G$, $x\in Z_G(L)$ if and only if $x\in Z_G(\mathfrak{l})$. 

Furthermore, $L$ contains the Cartan subgroup $H$, which is its own centralizer. It follows that if $x\in Z_G(L)$ then $x\in Z_G(H)=H$. Therefore, we may write any element $x\in Z_G(L)$ as $x=z\exp_G(\lambda_1H_1+\cdots+\lambda_m H_m)$ where $\lambda_1,\dots,\lambda_m\in \mathbb{C}$ and $z\in Z(G)$. 

Finally, any $x$ of this form will commute with $\mathfrak{h}$ and for $\alpha=n_1\alpha_1+\cdots+n_m\alpha_m\in \Delta$, we have 
\[\Ad(x)\cdot X_{\alpha}=\exp (\alpha(\lambda_1H_1+\cdots+\lambda_mH_m))X_{\alpha}=\exp(n_1\lambda_1+\cdots+n_m\lambda_m)X_{\alpha}.\]
Thus in the case $\alpha = \alpha_j$, we see that in order to have $\mathrm{Ad}(x)\cdot X_{\alpha_j}=X_{\alpha_j}$, we need $\lambda_j\in 2\pi \sqrt{-1}\mathbb{Z}$. 

It follows that $x=z\exp_G(X) \in \mathfrak{h}$ belongs to $Z_G(L)$ if and only if $X\equiv \lambda H_i$ in $\mathfrak{h}/\langle 2\pi\sqrt{-1}H_j\mid 1\leq j\leq r\rangle_{\mathbb{Z}}$, where $\langle 2\pi\sqrt{-1}H_j\mid 1\leq j\leq r\rangle_{\mathbb{Z}}$ is the additive subgroup of $\mathfrak{h}$ generated by the elements $2\pi\sqrt{-1}H_j$.

It follows that $Z_G(L)$ is generated by $Z(G)$ and a $1$-parameter subgroup $t\mapsto \exp(tH_i)$, whence $Z_G(L)/Z(G)$ is isomorphic to $\mathbb{C}^*$. 

We now denote for $n\neq 0$,
\[\mathfrak{u}_n:=\bigoplus_{\substack{\alpha\in \Delta\\ \alpha=n\alpha_i+\sum_{j\neq i}n_j\alpha_j}}\mathfrak{g}_{\alpha}.\]

Then it is clear that 
\[\mathfrak{g}=\mathfrak{l}\oplus \bigoplus_{n\neq 0}\mathfrak{u}_n.\]
Furthermore, with respect to $\mathbb{C}^*\cong Z_G(L)/Z(G)$, $\mathbb{C}^*$ acts trivially on $\mathfrak{l}$ and for $\lambda\in \mathbb{C}^*$ and $v\in \mathfrak{u}_n$, we have $\lambda\cdot v=\lambda^n v$. 

Now $\mathrm{Stab}(\rho)\leq PG$ is exactly $Z_G(L)/Z(G)\cong\mathbb{C}^*$, so the infinitesimal action on $H^1(\F_r;\mathfrak{g}_{\Ad_\rho})$ is given by the action (as above) on the corresponding coefficients of: 
\[H^1(\F_r;\mathfrak{g}_{\Ad_\rho})=H^1(\F_r;\mathfrak{l}_{\Ad_\rho})\oplus \bigoplus_{n\neq 0}H^1(\F_r;(\mathfrak{u}_n)_{\Ad_\rho}),\]
that is, $\mathbb{C}^*$ acts trivially on $H^1(\F_r;\mathfrak{l}_{\Ad_\rho})$ and acts on $H^1(\F_r;(\mathfrak{u}_n)_{\Ad_\rho})$ by $\lambda\cdot v=\lambda^nv$. 

By Lemma \ref{sing-lem}, $[\rho]$ is ugly if and only if $[0]$ is a topological singularity in $$H^1(\F_r;\mathfrak{g}_{\Ad\rho})\aq\mathrm{Stab}(\rho),$$ and we have established that, for some $N>0$,  
$$H^1(\F_r;\mathfrak{g}_{\Ad_\rho})\aq \mathrm{Stab}(\rho)\cong \left(\bigoplus_{n=-N}^{N}\mathbb{C}^{d_n}\right)\aq\mathbb{C}^*\cong \C^{d_0}\cross \left(\bigoplus_{n\neq 0}\C^{d_{n}}\right)\aq\mathbb{C}^*,$$ as in Lemma \ref{weighted-lem} below (we include  $H^1(\F_r;\mathfrak{l}_{\Ad_\rho})$ as the factor $\C^{d_0}$).

Now by direct computation we have:
\begin{align*}
\dim_\C H^1(\F_r;(\mathfrak{u}_n)_{\Ad_\rho})&=\dim_\C Z^1(\F_r,(\mathfrak{u}_n)_{\Ad_\rho})-\dim_\C B^1(\F_r,(\mathfrak{u}_n)_{\Ad_\rho})\\
&=\dim_\C Z^1(\F_r,(\mathfrak{u}_n)_{\Ad_\rho})-\dim_\C \mathfrak{u}_n+\dim_\C Z^0(\F_r,(\mathfrak{u}_n)_{\Ad_\rho})\\
&=r\dim_\C \mathfrak{u}_n-\dim_\C \mathfrak{u}_n+\dim_\C \mathfrak{u}_n^{\rho(\F_r)}\\
&=(r-1)\dim_\C \mathfrak{u}_n,\end{align*}
because the Zariski-closure of $\rho(\F_r)$ contains a Cartan subgroup of $G$ which has no non-zero fixed point on $\mathfrak{u}_n$. Since $\dim_\C \mathfrak{u}_n\geqs 1$, Lemma \ref{weighted-lem} finishes the proof so long as $r\geqs 3$. 

Finally, if $r=2$, it suffices to prove that  $\sum_{n\geq 1}\dim_\C \mathfrak{u}_n\geqs 2$ provided that $\mathfrak{g}$ has no simple factor of rank $1$. This is equivalent to the fact that $\mathfrak{g}$ has at least two positive roots. One may easily check this for any simple complex Lie algebra $\mathfrak{g}$ which is not $\mathfrak{sl}_2$ and therefore we may again apply Lemma \ref{weighted-lem} to complete the proof. 
\end{proof}

\begin{rem}
The first paragraph in the proof of Theorem \ref{redugly-thm}, that the ugly locus is closed, proves \cite[Conjecture 4.8]{FL2} to be true since \cite{FL2} shows that a Euclidean dense set of reducibles is ugly in the cases considered in \cite[Conjecture 4.8]{FL2}.
\end{rem}

\begin{rem}
We note that with respect to the actual application of Lemma \ref{weighted-lem} to Theorem \ref{redugly-thm}, that $d_1\geqs 1$ since there is always a simple root with eigenvalue $1$.  Moreover, the situation when $\sum_{n\geqs 1} d_n=1$ does in fact occur in $$\hom(\F_2,\SL_2(\C))\aq \SL_2(\C)\cong \C^3,$$ see \cite{FL2}.
\end{rem}

A priori it is not clear if bad representations are ugly or not.  Such a representation has a local model that is a finite group quotient of affine space.  By the Chevalley-Shephard-Todd Theorem \cite{ShTo,Che}, the quotient is smooth if and only if the finite group is generated by pseudoreflections.  But such a quotient, even if singular, can be topologically a manifold as the next example shows.  

\begin{exam}\label{starr-ex}
Let $\Gamma$ be the binary icosahedral group; that is, the group of symmetries of the icosahedron $($this group is of order 120, and is isomorphic to $\mathrm{SL}_2 (\mathbb{F}_5)$$)$.  The rotational symmetries  of the icosahedron  are naturally a subgroup of $\SO(3)$ and $\Gamma$ is the inverse image of this subgroup under the double covering $\SU(2) \to \SO(3)$. We call the inclusion 
$$\Gamma\injects \SU(2) \injects \GL_2 (\C)$$ 
 the ``standard representation.'' Now consider the homomorphism $\alpha:\Gamma\to \GL_3(\C)$ equal to the direct sum of the standard representation of $\Gamma$ with a trivial representation.  Then $\alpha$ defines an action on $\C^3$, and the quotient $\C^3/\Gamma$ is a normal, topological manifold that is algebraically singular \cite[Jason Starr's example]{MOsing}.
\end{exam}

We now show that bad representations are ugly if $r\geqs 3$, which implies the situation of the above example does not occur for the varieties $\X_r(G)$.  

\begin{thm}\label{badugly-thm}
Let $G$ be a connected, reductive $\C$-group, and assume either $r\geqs 3$, or $r\geqs 2$ and the Lie algebra of $DG$ has only simple factors of rank 2 or more. Let $[\rho]\in \X_r(G)$. If $\rho$ is bad, then $\rho$ is ugly.
\end{thm}

\begin{proof}
We repeat the set-up from Lemma \ref{sing-lem}.  Let $[\rho]\in \X_r(G)$ be bad; we assume that $\rho$ is polystable (in fact, stable).  There is a local model of $[\rho]$ of the form $V\aq \Gamma$, where $V:=H^1(\F_r;\mathfrak{g}_{\Ad_\rho})$ is a $\C$-vector space, $[0]$ corresponds to $[\rho]$, and $\Gamma:=\mathrm{Stab}(\rho)$ is a finite subgroup of $PG$ acting linearly (and so fixes $0$).

Since the ugly locus is closed and the bad locus is closed, it suffices to show that generic bad representations are ugly.  The set of bad representations with abelian stabilizer is generic by Proposition \ref{ab-gen}, so we will assume $\Gamma$ is abelian.  
Under the assumptions of the theorem, Proposition \ref{nopseudo} now shows that the action of $\Gamma$ on $V$ does not include any pseudoreflections.  We also note that if $\Gamma$ does not act effectively (faithfully), then $\Delta:=\cap_{v\in V}\stab_\Gamma(v)$ is a normal subgroup and $V\aq\hat{\Gamma}\cong V\aq \Gamma$ where $\hat{\Gamma}:=\Gamma/\Delta$. Since we have shown that bad representations are singular, we know $\Gamma$ cannot act trivially on $V$.  So without loss of generality, we may assume $\Gamma$ is a non-trivial abelian group, acting effectively.

Let $F \subset V$ denote the set of vectors fixed by some non-trivial element of $\Gamma$; this is a union of linear subspaces, invariant under the action of $\Gamma$. Since $\Gamma$ does not act by pseudoreflections, $F$ has codimension at least 2.

We now prove that $V\aq \Gamma$ has a topological singularity at $[0]$.  In what follows we replace $\aq$ with $/ $ since $\Gamma$ is finite and hence these quotients are equivalent.  Say $\dim_\C (V) = n$. For a space $X$, let $X^+$ denote its one-point compactification. Then $V^+ \homeo S^{2n}$, and since $F$ is a union of linear subspaces, $F^+$ is  a union of spheres (of dimension at most $2n-4$).

We claim $(\star)\ F^+/\Gamma \isom (F/\Gamma)^+$ is locally contractible, and $(\star\, \star)$ has no integral homology in dimensions greater than $2n-3$.  We prove $(\star)$ and $(\star\, \star)$ below, but first let us see how they are used.

Since $\Gamma$ acts freely on the complement of $F$, the quotient map $V - F \to V/\Gamma - F/\Gamma$ is a covering map, and $V - F$ is simply connected because $F$ is a union of smooth submanifolds of $V$, each with real codimension greater than 2. So $\pi_1 (V/\Gamma - F/\Gamma) = \Gamma$, and since $\Gamma$ is abelian, $H_1 (V/\Gamma - F/\Gamma) = \Gamma$ as well.

We have $(V/\Gamma)^+ - (F/\Gamma)^+ = V/\Gamma - F/\Gamma$. Assume $F \neq \{0\}$; the case $F = \{0\}$ is easier and treated at the end. If $V/\Gamma$ is a topological manifold around $[0]$, then by Proposition~\ref{sing-lem}, $V/\Gamma$ is homeomorphic to a Euclidean space; by dimension count,
we must in fact have
$V/\Gamma\homeo V$,  and so $(V/\Gamma)^+ \homeo S^{2n}$. By $(\star)$, we may apply Alexander Duality, yielding:
$$H_1 (V/\Gamma - F/\Gamma) \cong H_1 ((V/\Gamma)^+ - (F/\Gamma)^+) \cong H^{2n-2} ((F/\Gamma)^+).$$ This contradicts the $(\star\, \star)$, since the left hand side is non-zero but the right hand side is zero.

If $F = \{0\}$, then $\Gamma$ acts freely on $V - \{0\} \heq S^{2n-1}$ and so $\pi_1 (V/\Gamma - \{[0]\})\isom \Gamma$. This is also a contradiction, because $V/\Gamma - \{[0]\} \homeo \C^n - \{0\} \heq S^{2n-1}$.

Therefore, we have shown that if $\rho$ is a bad representation such that $\stab(\rho)$ is abelian and acts on $V$ without pseudoreflections, then $\X_r (G)$ has a topological singularity at $[\rho]$.  Since such representations form a generic subset of the bad locus, we have shown that bad representations are ugly (under the hypotheses of the theorem).

Proof of $(\star)$: Local contractibility at all points other than $\infty$ is immediate since $F/\Gamma$ is an algebraic subset of $V/\Gamma$, hence triangulable (see \cite{Hofmann}).

Let $|-|$ be a $\Gamma$-invariant norm on $V$ (as in the proof of Lemma~\ref{sing-lem}). We claim that $U_N:=F^+ \cap \{v\in V : |v| > N\}$ admits a $\Gamma$-equivariant deformation retraction to $\infty\in F^+$, which then descends to a deformation retraction of $U_N/\Gamma$ to $[\infty]\in F^+/\Gamma$ (giving the desired contractible neighborhoods around $[\infty]$). The deformation retraction is just given by sending $v \mapsto (1/(1-t))v$ at time $t$. This is $\Gamma$-equivariant since $\Gamma$ acts linearly.

Proof of $(\star\, \star)$: We study the homology of $(F/\Gamma)^+$ using the Mayer-Vietoris sequence for the open cover consisting of $F/\Gamma$ and $(F/\Gamma)^+ - \{[0]\}$. The latter set is contractible (as proven in the previous paragraph), so it will suffice to show that $F/\Gamma$ and $$F/\Gamma \cap ( (F/\Gamma)^+ - \{[0]\} ) = F/\Gamma - \{[0]\}$$ each have no integral homology in dimensions greater than $2n-4$. By the Universal Coefficient Theorem, it is enough to check this in cohomology. Both spaces are locally contractible, so their integral cohomology agrees with their \v{C}ech cohomology, and thus it is enough to verify that these spaces have topological covering dimension at most $2n-4$ (see \cite{gentop} for a discussion \v{C}ech cohomology and covering dimension).

But $F$ is a simplicial complex of dimension at most $2n-4$, and finite quotients do not increase covering dimension (\cite[Proposition 9.2.16]{Pears}) so $F/\Gamma$ also has covering dimension at most $2n - 4$, as does its open subspace $F/\Gamma- \{[0]\}$. This completes the proof of $(\star\, \star)$.
\end{proof}

We note that the binary icosahedral group has trivial abelianization, so the arguments in the proof of  Theorem~\ref{badugly-thm} do not apply to Example~\ref{starr-ex}.

\begin{rem}
In \cite{FLR}, it is shown that $\pi_2(\X_r(G))=0$ if $DG$ has its Lie algebra isomorphic to a product of special linear groups.  Given the results in Section \ref{Schur} this shows that if $G$ is a CI group, then $\pi_2(\X_r(G))=0$.

As in the proof of the above theorem, let $[\rho]\in \X_r(G)$ be bad and consider a local model of $[\rho]$ of the form $V\aq \Gamma$, where $V$ is a $\C$-vector space, $[0]$ corresponds to $[\rho]$, and $\Gamma$ is a finite subgroup of $PG$ acting linearly. The Topological Form of Zariski's Main Theorem $($\cite{Redbook}$)$ implies any path-connected open neighborhood around $[0]$ has its smooth part $($the complement of the singular locus$)$ path-connected.  Let $S$ be the singular locus in $V\aq \Gamma$ $($which has codimension at least 2 by normality$)$, and let $p:V\to V\aq \Gamma$ be the quotient map $($which has path-lifting by \cite{LR}$)$.  By Proposition \ref{nopseudo} the set of points where $\Gamma$ does not act locally injective is codimension at least 2, and so by Lemma \ref{NR-lem}, $p^{-1}(S)$ is this collection.

Let $U$ be a path-connected open neighborhood of $[0]$.  Since $\{0\}=p^{-1}([0])$ and $[0]\in U$, we conclude that $p^{-1}(U)$ is path-connected and so by codimension $W:=p^{-1}(U-S)$ is path-connected too.  Since $p$ is a quotient map, $W$ is saturated and so $\Gamma$ acts on $W$. It acts freely since we have removed the points that map to singularities.  Since $\Gamma$ is finite and so discrete, we have that $p:W\to U-S$ is a non-trivial covering map.  Thus, $U-S$ cannot be simply connected. 

Thus, the proof in \cite{FLR} that $\pi_2(\X_r(\SL_n(\C)))=0$ cannot generalize to other reductive $\C$-groups $G$ if $G$ is not CI since the proof in \cite{FLR} required the existence of neighborhoods, in particular around bad representations, whose smooth locus was simply connected $($which we just showed is impossible if $G$ is not CI since in that case there will always be bad representations$)$.
\end{rem}

\section{Homotopy Groups of Good Representations}

In this section we compute the homotopy groups of the good locus of $\X_r(G)$, 
in a range of dimensions tending to infinity with $r$, 
when $G$ is a connected, reductive $\C$-group and when $r\geqs 2$. As per our previous results, in most of these cases the smooth locus $\mathcal{X}_r(G)$ is equal to the good locus $\X_r(G)^{good}=\hom(\F_r,G)^{good}/G$.

\begin{rem}
We exclude the $r=1$ case, since if $r=1$ and $G$ is non-abelian, $\X_1(G)\cong T/W\times_F Z$ where $T$ is a maximal torus in $DG$, $W$ is its Weyl group, $Z$ is the center in $G$, and $F=DG\cap Z$.  In this case, $Z$ is a complex torus and $T/W$ is contractible.  Moreover, the irreducible locus and hence the good locus is empty in this case. 

When $G$ is abelian, the topology is also understood. In this case, $\X_r(G)\cong (\C^*)^{rn}$ since $G\cong (\C^*)^n$, and all representations are good. The results in this section are trivially true when $G$ is abelian, and so in our proofs we only treat the non-abelian case.
\end{rem}

We begin by reviewing some lemmata in \cite{FLR, BL} that we will find useful.

\begin{lem}[Lemma 4.4 in \cite{FLR}]\label{nullhomotopy} Let $G$ be a connected Lie group, and assume $r\geqs 2$.  Then for each $\rho \in G^r$, the map $PG\to G^r$ given by $[g]\mapsto [g]\rho[g]^{-1}$ is nullhomotopic.
\end{lem}

\begin{lem}[Lemma 2.2 in \cite{BL}]\label{lem22BL}
Let $G$ be a connected, reductive $\C$-group, and assume $r\geqs 2$. Then $$\hom(\F_r,G)^{good}\to \XC{r}(G)^{good}$$
is a principal $PG$-bundle.
\end{lem}

Lemma \ref{nullhomotopy} is proven by considering a path $\rho_t$ from $\rho$ to the trivial representation. Lemma \ref{lem22BL} follows from the fact that the conjugation action of $PG$ on the good locus is free and proper.

Now, we put together the main theorem from a previous section (Theorem \ref{codimbad}) with a generalization of \cite[Theorem 2.9]{FLR}.

\begin{lem}\label{lem-gencodim}
Let $G$ be a connected, reductive $\C$-group, and assume $r\geqs 2$.  Then 
$$\mathrm{codim}_\C\hom(\F_r,G)^{red}\geqs (r-1)\mathrm{Rank}(DG).$$
\end{lem}

\begin{proof} This result is essentially contained in the proof of \cite[Theorem 2.9]{FLR}; we briefly outline the computation. First, one finds (similar to Lemma \ref{codim}) that 
 \begin{eqnarray*}\codim_\C \left(\hom(\F_r,G)^{red} \right)
\geqs  (r-1)\left(\dim_\C (G/P_{max})\right),
\end{eqnarray*} 
 where $P_{max}$ is a maximal dimensional proper parabolic subgroup of $G$.
 The Bruhat decomposition of the flag variety
$G/P_{max} \isom DG/(P_{max}\cap DG)$
shows  this quotient contains a maximal torus $T$ of $DG$, giving 
$$\codim_\C \left(\hom(\F_r,G)^{red} \right) \geqs (r-1)\left(\dim_\C (G/P_{max})\right)\geqs(r-1)\mathrm{Rank}(DG).$$
\end{proof}

Recall that by \cite[Proposition 1.3]{JM}, the set $\hom(\F_r,G)^{good}$  is Zariski open  in $\hom(\F_r,G)$, so its complement $\hom(\F_r, G)^{bad}\cup\hom(\F_r,G)^{red}$ is algebraic.
We define
\[C_{pasbon}:=\codim_\R\left(\hom(\F_r, G)^{bad}\cup\hom(\F_r,G)^{red}\right).\]  
The authors emphasize that the codimension in the definition of $C_{pasbon}$ is considered over $\R$. 
Combining Theorem \ref{codimbad} and Lemma~\ref{lem-gencodim} gives the following lower bound on $C_{pasbon}$.

\begin{thm}\label{gencodim} Let 
$R_G:=\min\{\mathrm{Rank}(G')\ |\ G' \text{ is a simple factor of } \widetilde{DG}\}$.
Then
$$C_{pasbon}\geqs 2\min\left\{2(r-1)R_G,(r-1)\mathrm{Rank}(DG)\right\}.$$

In particular, $C_{pasbon}\geqs 2$ if $r\geqs2$ and $G$ is non-abelian, and $C_{pasbon}$ grows unboundedly in $r$ and in the minimum rank of a simple factor of $[\mathfrak{g},\mathfrak{g}]$.
\end{thm}

We now turn our attention to homotopy groups, beginning with $\hom(\F_r,G)^{good}$.

\begin{lem}\label{homotopy-codim} Assume $r\geqs 2$.  The inclusion map induces an isomorphism 
$$\pi_k  \left(\hom(\F_r, G)^{good}\right) \stackrel{\isom}{\maps}  \pi_k  \left(\hom(\F_r, G)\right)\isom \pi_k (G)^r$$ for $k \leqs C_{pasbon} -2$, and is a surjection for $k=C_{pasbon} -1.$
\end{lem}

\begin{proof}  Since $\hom(\F_r, G)^{bad}\cup\hom(\F_r,G)^{red}$ is Zariski closed in $\hom(\F_r,G)$, it is a finite union of  locally closed submanifolds, and the codimension $C_{pasbon}$ is a lower bound on the real codimension  these submanifolds.
Since $\hom(\F_r,G)\cong G^r$ is a smooth manifold, transversality shows that every map $S^k \to \hom(\F_r,G)$ with  $k\leqs C_{pasbon}-1$  is homotopic to a map with image in the good locus, and for $k\leqs C_{pasbon}-2$ every homotopy between such maps can be deformed into the good locus.\footnote{Our use of transversality in this context is analogous to \cite[Corollary 4.8]{Ramras}.}
\end{proof} 

We now combine Lemmas~\ref{nullhomotopy} and \ref{homotopy-codim}.

\begin{lem}\label{fiber-null}   Let $r\geqs 2$.  For any $\rho\in \hom(\F_r, G)^{good}$, the orbit-inclusion map $PG \to \hom(\F_r, G)^{good}$, $[g] \mapsto [g]\rho [g]^{-1}$, induces the zero map on homotopy groups in dimensions at most $C_{pasbon}-2$. 
 \end{lem}

\begin{proof} By Lemma~\ref{nullhomotopy}, the composite map
$$PG \to \hom(\F_r, G)^{good} \to \hom(\F_r, G)$$
is nullhomotopic, and by Lemma \ref{homotopy-codim}, the second map in this composition is an isomorphism in the stated range.
\end{proof}

\begin{thm}\label{pi0pi1}
Let $r\geqs 2$ and $G$ a connected, reductive $\C$-group.  Then $\pi_0(\X_r(G)^{good})$ is trivial, and if $r\geqs 3$ or the rank of $DG$ is at least 2, then $$\pi_1(\X_r(G)^{good})\cong \pi_1(G)^r$$ and  $\pi_2(\X_r(G)^{good})\cong \pi_1 (PG)$.
\end{thm}

\begin{proof}
Since $\hom(\F_r,G)\cong G^r$ is irreducible, then so is $\hom(\F_r,G)^{good}$ as it is Zariski open.  Thus, $\X_r(G)^{good}$ is irreducible too and hence connected.

Since $\hom(\F_r,G)^{good}\to \X_r(G)^{good}$ is a principal $PG$-bundle, there is a long exact sequence in homotopy:  
\begin{eqnarray*}
\cdots \to \pi_2 (PG)\to \pi_2(\hom(\F_r,G)^{good})\to \pi_2(\X_r(G)^{good})\hspace{1.5in}\\
\to \pi_1(PG)\to \pi_1(\hom(\F_r,G)^{good})
 \to \pi_1(\X_r(G)^{good})\to \pi_0(PG)=0.\end{eqnarray*}
Lemma \ref{fiber-null} and Lemma \ref{homotopy-codim} then imply $\pi_1(\X_r(G)^{good})\cong \pi_1(G)^r$ as long as $1\leqs C_{pasbon}-2$. Since $\pi_2 (G) = 0$, the same reasoning shows that 
$\pi_2(\X_r(G)^{good}) \isom \pi_1 (PG)$ when $1\leqs C_{pasbon}-2$. By Theorem~\ref{gencodim}, this bound holds if either $r\geqs 3$ or the rank of $DG$ is at least 2. 
\end{proof}

\begin{rem}
The fundamental group of $G$, always abelian, is the same as that of its maximal compact subgroup $K$ since $G$ deformation retracts to $K$.  A standard result $($see \cite{Hall}$)$ gives the fundamental group of $K$.  Precisely, let $\mathfrak{k}$ be the Lie algebra of $K$ and let $\mathfrak{t}$ be a maximal commutative subalgebra of $\mathfrak{k}$. Then $\pi_1(K)\cong \Gamma/\Lambda$, where $\Gamma$ is the kernel of the exponential mapping for $\mathfrak{t}$ and $\Lambda$ is the lattice generated by the real co-roots.  Thus, the fundamental group of $\X_r(G)^{good}$ is isomorphic to $(\Gamma/\Lambda)^r$ under the conditions of Theorem \ref{pi0pi1}.
\end{rem}

With the first few homotopy groups computed for the good locus, we now turn our attention to the higher homotopy groups.

By work in \cite[Theorem 3.3]{FLR} it suffices to consider the case where $G$ is semisimple, since $$\pi_k(\X_r(G)^{good})\cong \pi_k(\X_r(DG)^{good})$$ for $k\geqs 2$.  We also note that $\pi_k(PG)=\pi_k(G)$ for $k\geqs 2$ since $G\to PG$ is a fibration whose fiber (an algebraic torus cross a finite group) is $\pi_k$-trivial for $k\geqs 2$.  This then implies $\pi_k(G)=\pi_k(DG)$ for $k\geqs 2$ since $PG=PDG$.

\begin{thm}\label{splitting} Let $r\geqs 2$.  Assume $1 \leqs k \leqs C_{pasbon}-2$. Then $$\pi_k(\X_r(G)^{good})\cong \pi_k(G)^r\times \pi_{k-1}(PG).$$
\end{thm} 

\begin{proof}
Since $\hom(\F_r, G)$ is an irreducible algebraic set, every non-empty Zariski open subset of $\hom(\F_r, G)$ is path connected, and it follows that $\XC{r}(G)^{good}$ is also path connected.  
 
By Lemma \ref{lem22BL},
$$PG\to \hom(\F_r, G)^{good}\to \XC{r}(G)^{good}$$ 
is a $PG$-bundle (where the first map is the inclusion of an adjoint orbit).  Hence we have an exact sequence
\begin{equation} \label{es}\cdots \to \pi_1(PG)\to \pi_1(\hom(\F_r, G)^{good})\to \pi_1(\XC{r}(G)^{good})\to 0.\end{equation}

When $2 \leqs k \leqs C_{pasbon}-2$, Lemma \ref{fiber-null} implies the long exact sequence breaks into short exact sequences: $$0\to \pi_k(\hom(\F_r,G)^{good})\to\pi_k(\X_r(G)^{good})\to \pi_{k-1}(PG)\to 0.$$  Lemma \ref{homotopy-codim} implies  
$\pi_k(\hom(\F_r,G)^{good})\cong \pi_k(G)^r$ and so we can write the short exact sequences as: 
\begin{equation}\label{eq:SES} 0\to \pi_k(G)^r\to\pi_k(\X_r(G)^{good})\to \pi_{k-1}(PG)\to 0.
\end{equation}
By Proposition \ref{good-htpy} in the following Section \ref{splitsection}, these short exact sequences split (non-canonically) if $1\leqs k\leqs C_{pasbon}-2$.

\end{proof}

\begin{cor}\label{pi34-cor}
Assume $C_{pasbon}\geqs 6$ and $r\geqs 2$,
and let $s$ be the number of simple factors of the Lie algebra of $DG$, and $t$ the number of those factors of type $A_1,B_1,$ or $C_n$ for $n\geqs 1$.  Then $\pi_3(\X_r(G)^{good})\cong \Z^{sr}$ and $\pi_4(\X_r(G)^{good})\cong (\Z_2)^{rt}\times \Z^s.$
\end{cor}

\begin{proof}
The universal cover of $DG$ has the form $\prod_{i=1}^s G_i$ where each $G_i$ is simple and simply connected, and $s$ is the number of simple factors of the Lie algebra of $DG$. Since $\pi_k(G)\cong \pi_k(DG)$ for $k\geqs 2$, we conclude that $\pi_k(G)\cong \oplus_i\pi_k(G_i)$ for $k\geqs 2$.

By \cite{Bott56}, $\pi_3(G)=\Z$ if $G$ is simple. By \cite{BottSam}, $\pi_4(G)=0$ or $\pi_4(G)\cong \Z_2$ if $G$ is simple; it is $\Z_2$ exactly when $G$ is of type $A_1$, $B_1$, or $C_n$ for $n\geqs 1$ (and 0 otherwise). 

Thus, because of the splitting in Theorem \ref{splitting}, we have that $$\pi_3(\X_r(G)^{good})\cong \Z^{sr},$$ and $$\pi_4(\X_r(G)^{good})\cong (\Z_2)^{rt}\times \Z^s$$ where $t$ is the number of simple factors of type $A_1,B_1,$ or $C_n$ for $n\geqs 1$ in the Lie algebra of $DG$.

\end{proof}

Note that by Theorem \ref{gencodim}, if $r\geqs 4$, then $C_{pasbon}\geqs 6$.

\begin{rem}
Let $r\geqs 3$. Then the good locus is the smooth locus, and so the above theorem and corollary are true when replacing $\X_r(G)^{good}$ by the smooth locus $\mathcal{X}_r(G)$.
\end{rem}

\begin{rem}\label{cor-per}
By Theorems \ref{gencodim} and \ref{splitting}, if $G$ is simple then
$\pi_k(\X_r(G)^{good})\cong \pi_k(G)^r\times \pi_{k-1}(PG)$ for $1\leqs k\leqs 2(r-1)\mathrm{Rank}(G)-2$. Given Bott periodicity for the classical groups $A_n$, $B_n$, $C_n$, and $D_n$ \cite{bott-per}, 
it follows that $\X_r(G)^{good}$ also exhibits periodic homotopy within appropriate stable ranges for the classical groups.
\end{rem}

See Example \ref{ex-per} for a more precise formulation of Remark \ref{cor-per}.

\begin{exam}
By similar reasoning as used in the proof of Corollary \ref{pi34-cor}, if  $r\geqs 2$ and $C_{pasbon}\geqs 7$ $($which holds if $r\geqs 7$$)$,
 we conclude that: $$\pi_5(\X_r(G)^{good})\cong \bigoplus_i^m\pi_5(G_i)^r\oplus (\Z_2)^t,$$ where  $t$ is the number of simple factors of type $A_1,B_1,$ or $C_n$ for $n\geqs 1$ in the universal cover of $DG$, itself a product of the simple simply connected $\C$-groups $G_1,...,G_m$.  By results in \cite{bott-per, BottSam, MimTod, MimTod-Sym, Mim-G2F4} we can calculate the fifth homotopy group of all simple $G$.  In particular, if $G_i$ is exceptional, then $\pi_5(G_i)=0$.  If $G_i$ is of type $A_n$, then $\pi_5(G_i)\cong \Z$ if $n\geqs 2$ and $\pi_5(G_i)\cong \Z_2$ if $n=1$.  If $G_i$ is of type $B_n$, then $\pi_5(G_i)\cong 0$ if $n\geqs 3$, and $\pi_5(G_i)\cong \Z_2$ if $n=1,2$.  If $G_i$ is of type $C_n$, then $\pi_5(G_i)\cong \Z_2$ for all $n\geqs 1$.  If $G_i$ is of type $D_n$, then $\pi_5(G_i)\cong 0$ if $n=1$ or $n\geqs 4$, $\pi_5(G_i)\cong \Z$ if $n=3$, and $\pi_5(G_i)\cong (\Z_2)^2$ if $n=2$.  So although it is not a clean formula, this completely describes the fifth homotopy groups of $\X_r(G)^{good}$.
\end{exam}

In short, if one knows the homotopy groups of $G$, then Theorem \ref{splitting} allows one to compute the $k$-th homotopy groups of $\X_r(G)^{good}$ for sufficiently large $r$. As an example of this, we next list the $k$-th homotopy groups for $0\leqs k\leqs 15$ when $G$ is an exceptional Lie group.

\begin{exam}We consider the complex adjoint type of each exceptional Lie group below.   They are all simply connected except $E_6$ and $E_7$ with fundamental group $\Z_3$ and $\Z_2$ respectively.  Since they are of adjoint type, $G=PG$ in each case.  

We assume that $r\geqs 2$ generally.  However, if a cell is highlighted red in Table \ref{splitting}, then we have assumed $r\geqs 3$, if it is highlighted orange then we have assumed $r\geqs 4$, if it is highlighted yellow then we have assumed $r\geqs 5$, and if it is highlighted green, then we have assumed $r\geqs 6$.

An ``?'' in a cell of the table means that the homotopy groups needed for the computation are not known $($as far as we know$)$.  Although for $E_6$, $E_7$, and $E_8$, for the cases where there is an ? and beyond, the $2$ and $3$-primary parts of the homotopy groups are known; see \cite{Kachi, KaMi}.  So one can obtain corresponding facts about the homotopy groups of $\X_r(G)^{good}$ in these cases.

Our main sources of reference for the computations in Table \ref{splitting}, aside from Theorem \ref{splitting}, are \cite{BottSam, Mim-G2F4}.

\begin{table}[!ht] %\cellcolor{blue!25}text
\begin{tabular}{c||c|c|c|c|c}
$k\setminus G$&$G_2$&$F_4$&$E_6$&$E_7$&$E_8$\\ \hline \hline
$0$ & $0$& $0$&$0$ &$0$ &$0$ \\ \hline
$1$ & $0$&$0$ &$(\Z_3)^r$ &$(\Z_2)^r$ & $0$\\ \hline
$2$ &$0$ &$0$ &$\Z_3$ &$\Z_2$ & $0$\\ \hline
$3$ &\cellcolor{red!25}$\Z^r$ & $\Z^r$& $\Z^r$& $\Z^r$&$\Z^r$ \\ \hline
$4$ & \cellcolor{red!25}$\Z$& $\Z$& $\Z$& $\Z$&$\Z$ \\ \hline
$5$ &\cellcolor{red!25}$0$ &$0$ & $0$&$0$ & $0$\\ \hline
$6$ &\cellcolor{red!25}$(\Z_3)^r$ & $0$&$0$ &$0$ &$0$ \\ \hline
$7$ &\cellcolor{orange!25}$\Z_3$ &\cellcolor{red!25} $0$& $0$& $0$&$0$ \\ \hline
$8$ &\cellcolor{orange!25}$(\Z_2)^r$ &\cellcolor{red!25} $(\Z_2)^r$ &$0$ &$0$ &$0$ \\ \hline
$9$ &\cellcolor{orange!25}$(\Z_6)^r\oplus\Z_2$ &\cellcolor{red!25}$(\Z_2)^{r+1}$ & $\Z^r$& $0$& $0$\\ \hline
$10$ &\cellcolor{orange!25}$\Z_6$ &\cellcolor{red!25}$\Z_2$ & ?& $0$&$0$ \\ \hline
$11$ &\cellcolor{yellow!25} $\Z^r\oplus (\Z_2)^r$&\cellcolor{red!25}$\Z^r\oplus (\Z_2)^r$ & ?&$\Z^r$  &$0$ \\ \hline
$12$ &\cellcolor{yellow!25}$\Z\oplus \Z_2$ &\cellcolor{red!25}$\Z\oplus \Z_2$ & ?& ?& $0$\\ \hline
$13$ &\cellcolor{yellow!25}$0$ &\cellcolor{red!25}$0$ & ?& ?& $0$\\ \hline
$14$ &$(\Z_{168})^r\oplus(\Z_2)^r$\cellcolor{yellow!25} &\cellcolor{red!25}$(\Z_2)^r$ & ?& ?& $0$\\ \hline
$15$ &\cellcolor{green!25} $(\Z_2)^{r+1}\oplus\Z_{168}$&\cellcolor{orange!25}$\Z^r\oplus \Z_2$ &? &? &\cellcolor{red!25} $\Z^r$\\ \hline
\end{tabular}
\caption{$\pi_k(\X_r(G)^{good})\cong \pi_k(G)^r\oplus\pi_{k-1}(PG)$}\label{exceptionalcases}
\end{table}
While the last row of Table \ref{exceptionalcases} stops at $k=15$, one can easily compute the homotopy groups up to $k=22$ for $G_2$ and $F_4$ using \cite{Mim-G2F4}.  For example, one finds that $\pi_{22}(\X_r(G_2)^{good})\cong \Z_{1386}\oplus\Z_8$ if $r\geqs 7$ and $\pi_{18}(\X_r(F_4)^{good})\cong \Z_{720}\oplus\Z_3$ if $r\geqs 4$.
\end{exam}

We now illustrate the periodicity that comes from Theorem \ref{splitting} for the classical groups $A_n$, $B_n$, $C_n$, and $D_n$ (Remark \ref{cor-per}).

\begin{exam}\label{ex-per}
For this example, we refer to \cite{bott-per}.  We assume $r\geqs 2$.  First, if $k\leqs n-2$ then $k\leqs 2(r-1)\mathrm{Rank}\left(\SO_n(\C)\right)-2$ which then implies $$k+8\leqs  2(r-1)\mathrm{Rank}\left(\SO_{n+8}(\C)\right)-2,$$ since 
$\mathrm{Rank}\left(\SO_n(\C)\right)=\left\{\begin{array}{ll}n/2,&\text{ if }n\text{ is even}\\ (n-1)/2,&\text{ if }n\text{ is odd.}\end{array}\right.$  Thus, if  $2\leqs k\leqs n-2$, then \begin{eqnarray*} \pi_k\left(\X_r(\SO_n(\C))^{good}\right)&\cong& \pi_k(\SO_n(\C))^r\oplus \pi_{k-1}(\SO_n(\C))\\&\cong& \pi_{k+8}(\SO_{n+8}(\C))^r\oplus \pi_{k+7}(\SO_{n+8}(\C))\\&\cong& \pi_{k+8}\left(\X_r(\SO_{n+8}(\C))^{good}\right).\end{eqnarray*}  So in particular, $\pi_{k}\left(\X_r(\SO_n(\C))^{good}\right)\cong \Z $ for all $k\equiv 7\mod 8$ and $n\equiv 9\mod 8$ so long as $k\geqs 7$ and $n\geqs 9$.

Likewise, there is $\pi_k$-periodicity in the $A_n$ series for $k\leqs 2n+1$ $($shift in $n$ is $+1$ and shift in $k$ is $+2)$ and $C_n$ series for $k\leqs 4n+1$ $($shift in $n$ is $+4$ and shift in $k$ is $+8)$. 

\end{exam}

On the other hand, our result shows that the homotopy groups can vary consistently in $r$ once $r$ gets sufficiently large.

\begin{exam}

From \cite{Mim-G2F4}, $\pi_{22}\left(\X_r(\SO_9(\C))^{good}\right)\cong (\Z_{11!/32})^r\oplus(\Z_8)^r\oplus(\Z_2)^{2r}\oplus \Z_{12}$ for all $r\geqs 4.$  There are many other examples along these lines.

\end{exam}

\section{Splitting short exact sequences}\label{splitsection}

The goal of this section is to prove the following proposition.

\begin{prop}~\label{good-htpy}
In dimensions $1\leqs k \leqs  C_{pasbon} - 2$, we have 
$$\pi_k \left(\XC{r}(G)^{good}\right) \isom \pi_k (G)^r \cross \pi_{k-1} (PG).$$
\end{prop}

The following algebraic fact now implies that the short exact sequences (\ref{eq:SES}) admit (non-canonical) splittings.

\begin{lem} If $0\to A\srt{i} B\srt{q} C\to 0$ is a short exact sequence of finitely generated abelian groups, and there exists an isomorphism $B\srt{\isom} A\cross C$, then the sequence splits.
\end{lem}
\begin{proof} 
This follows from the results in \cite{Miyata}.
\end{proof}

We now prepare for the proof of Proposition~\ref{good-htpy}, which will be at the end of this section.

Given a topological group $K$, we let $BK$ be its classifying space and $\pi\co EK \to BK$ denote a universal principal $K$-bundle; that is, a (right) principal $K$-bundle with $EK$ contractible. We note that there are at least two functorial constructions of $EK$: Milnor's infinite join construction (which works for all topological groups) and the standard simplicial bar construction (which works for all Lie groups). Either of these models will suffice for our purposes below.

\begin{defn} Let $K$ be a topological group, and let $X$ be a $($left$)$ $K$-space. Then the \e{homotopy orbit space} for the action of $K$ on $X$ is the space
$$X_{hK} := EK \cross_K X = (EK \cross X)/K,$$
where $K$ acts by $(e, x)\cdot k = (ek, k^{-1} x)$.
\end{defn}

We record some standard facts regarding homotopy orbit spaces.

There is  a natural map $p_X \co X_{hK}\to BK$, induced by the projection $EK\to BK$. This map is a fiber bundle with fiber $X$, locally trivial over each open set in $BK$ over which $EK$ is trivial.

The next fact may be found, for instance, in Atiyah-Bott~\cite[Section 13]{AB}.

\begin{lem}\label{free} Let $X$ be a $K$-space such that the projection map $X\to X/K$ is a principal $K$-bundle. Then the map $X_{hK} \to X/K$, sending $[(e, x)]\in EK\cross_K X$ to $[x]\in X/K$, is a weak homotopy equivalence.
\end{lem}

Recall that a map $f\co X\to Y$ is said to be $n$-connected if the induced map on homotopy groups is an isomorphism in degrees less than $n$ and is surjective in degree $n$. 

\begin{lem}\label{n-ctd} If $X\to Y$ is an equivariant map of $K$-spaces, and $f$ is $n$-connected, then so is the   map $f_{hK}\co X_{hK}\to Y_{hK}$ induced by $f$.
\end{lem}
\begin{proof} The lemma  follows by applying the Five Lemma to the 
diagram of long exact sequences in homotopy induced by the 
 commutative diagram
$$\xymatrix{X\ar[r]\ar[d]^f &X_{hK} \ar[r]^{p_X}  \ar[d]^{f_{hK}} & BK \ar[d]^=\\
Y \ar[r] &Y_{hK} \ar[r]^{p_Y}   & BK.
}$$
\end{proof}

With these lemmata complete, we now prove Proposition \ref{good-htpy}.
\begin{proof}[Proof of Proposition \ref{good-htpy}]
Consider the principal $PG$-bundle 
$$PG \maps \hom(\F_r,G)^{good} \maps \XC{r}(G)^{good}.$$
Lemma~\ref{free} shows that we have a weak equivalence
\begin{equation}\label{eq:good1}\XC{r}(G)^{good} \heq (\hom(\F_r,G)^{good})_{hPG}.\end{equation}
By Lemma~\ref{homotopy-codim}, the inclusion
$$\hom(\F_r,G)^{good}\injects G^r$$
is a $(C_{pasbon} - 1)$--connected map, so Lemma~\ref{n-ctd} implies the induced map 
\begin{equation}\label{eq:good2} (\hom(\F_r,G)^{good})_{hPG} \maps (G^r)_{hPG}
\end{equation}
is $(C_{pasbon} - 1)$--connected as well.
Since the identity element  $e\in G^r$ is fixed under conjugation, the fibration
\begin{equation}\label{eq:hPG}(G^r)_{hPG}\to BPG\end{equation}
admits a splitting, defined by $[x]\mapsto [x, e] \in EPG\cross_{PG} G^r$;  this splitting is continuous because the map $EPG \to  EPG\cross_{PG} G^r$, $x\mapsto [x,e]$, is continuous, and $BG \isom EPG/PG$ (as holds for all principal bundles). It follows that the long exact sequence associated to (\ref{eq:hPG}) breaks up into into \e{split} short exact sequences of the form
\begin{equation}\label{eq:hPG2}0 \maps \pi_k (G^r) \maps \pi_k\left( (G^r)_{hPG}\right) \maps \pi_k (BPG) \maps 0.
\end{equation}

For any Lie group $H$, we have $\pi_k (BH) \isom \pi_{k-1} (H)$ for each $k\geqs 1$, so the split short exact sequences (\ref{eq:hPG2}) yield
$$\pi_k  \left( (G^r)_{hPG}\right)\isom \pi_k (G)^r \cross \pi_{k-1} (PG)$$
for $k\geqs 1$. We saw above that the map (\ref{eq:good2}) is $(C_{pasbon} - 1)$-connected, so this completes the proof of Proposition~\ref{good-htpy}.
\end{proof}

\section{A generalization of Schur's lemma}\label{Schur}

In this section we will characterize connected, reductive $\C$-groups containing no bad subgroup. These are called CI-groups (see \cite{Si4} for definition). We will also give a rather simple description for the bad locus of character varieties in simply connected semisimple $\C$-groups. 

If $G$ is semisimple, let $\Lambda_G$ denote the lattice generated by the roots in the dual of the Lie algebra of a fixed Cartan subgroup of $G$ (see \cite[Chapter 23]{FH}). 

We begin with a lemma about the centralizers of BdS subgroups. We shall use the fact, which follows from the definition, that a BdS subgroup of $G$ is defined, up to conjugation, by a sub-root system of the root system of $G$ with identical rank. The lemma itself comes from Borel and de Siebenthal's article (done in the compact case but is essentially identical), see \cite{BdS}.

\begin{lem}\label{ZBdS}
Let $S$ be a BdS subgroup of a connected, reductive $\C$-group $G$. Then $Z_G(S)=Z(S)$ and furthermore, we have an isomorphism between 
\[Z_G(S)/Z(G)\text{ and } \Lambda_G/\Lambda_S.\]
\end{lem}
\begin{proof}
Let $H$ be a Cartan subgroup of $S$ (it is then a Cartan subgroup in $G$). Since $H$ is a Cartan subgroup of $G$, $Z_G(H)=H$ whence $Z_G(S)\leqs Z_G(H)\leqs S$, therefore $Z_G(S)=Z(S)$. 

Fix $\alpha_1,\dots, \alpha_r$ a system of simple roots for $G$ and $\beta_1,\dots, \beta_r$ a system of simple roots for $S$. Let $\Gamma_G$ (respectively $\Gamma_S$) be the lattice of elements in $\mathfrak{h}$ which are sent to integers via the functionals in $\Lambda_G$ (respectively $\Lambda_S$). 

Using the functoriality of the exponential map, one sees that an element $h=\exp(X)\in H$ will commute with all elements in $G$ (respectively $S$) if and only if $X$ belongs to $2\sqrt{-1}\pi \Gamma_G$ (respectively $2\sqrt{-1}\pi \Gamma_S$). Whence $Z_G(S)/Z(G)$ is isomorphic to $2\sqrt{-1}\pi \Gamma_S/2\sqrt{-1}\pi \Gamma_G$ which is isomorphic to $\Gamma_S/\Gamma_G$. 

Finally, since $\Lambda_G/\Lambda_S$ is finite, there is a perfect pairing between $\Lambda_G/\Lambda_S$ and $\Gamma_S/\Gamma_G$ induced by the perfect pairing $\mathfrak{h}^*\times \mathfrak{h}\to \mathbb{C}$. In particular,  $\Gamma_S/\Gamma_G$ is isomorphic to  $\Lambda_G/\Lambda_S$.\end{proof}

We immediately deduce the following corollary:
\begin{cor}
Any BdS subgroup in a connected, reductive $\C$-group is bad.
\end{cor}
\begin{proof}
Let $G$ be a connected, reductive $\C$-group and $S$ be a BdS subgroup. Because of the preceding lemma, if we had $Z_G(S)=Z(G)$ then we would have $\Lambda_G=\Lambda_S$. This is impossible because this would imply that $G=S$. As a result, $Z_G(S)\neq Z(G)$. 

It is a routine verification to show that if $L$ is a Levi subgroup of a parabolic subgroup of $G$ then $\dim_\C Z_G(L)>\dim_\C Z(G)$.  If $S$ were contained in a parabolic subgroup then it would be contained in one of its Levi subgroups since $S$ is reductive and we would therefore have $\dim Z_G(S)>\dim Z(G)$. The preceding lemma implies that $\dim Z_G(S)=\dim Z(G)$, and so $S$ is not contained in any parabolic subgroup of $G$. Therefore, $S$ is irreducible. \end{proof}

It is easy to see that $\SL_n(\C)$ and $\GL_n(\C)$ are CI by Schur´s lemma (see \cite[Lemma 3.5]{FL2}), and $\mathrm{O}_n(\C),\Sp_{2n}(\C), \p\SL_n(\C)$ are not CI (see \cite[Proposition 3.32]{FL2}).  In \cite[Question 19]{Si4}, Sikora asks: Are $\GL_n(\C)$ and $\SL_n(\C)$ the only CI-groups?  We now give a characterization of such groups, answering Sikora's question.

\begin{thm}\label{CI-thm}
A connected, reductive $\C$-group is a CI-group if and only if its derived subgroup is a product of special linear groups. 
\end{thm}

\begin{proof}
First, notice that if $\pi:G_1\to G_2$ is a finite covering of connected, reductive $\C$-groups and $S$ is a bad subgroup of $G_1$ then $\pi(S)$ is a bad subgroup of $G_2$. 

Secondly, for any connected, reductive $\C$-group $G$ there is a finite cover $Z(G)\times [G,G]\to G$ sending $(z,s)$ to $zs$. If $S$ is a bad subgroup of $[G,G]$ then it is a bad subgroup of  $Z(G)\times [G,G]$ whence it is a bad subgroup of $G$. On the other hand, if $S$ is a bad subgroup of $G$ then $\langle Z(G),S\rangle\cap [G,G]$ is a bad subgroup of $[G,G]$. So that $G$ contains a bad subgroup if and only if $[G,G]$ contains a bad subgroup. As a result, it suffices to show that among semisimple groups, the only CI-groups are the ones that are products of $\SL_n(\mathbb{C})$ for potentially varying $n$. 

If $G$ is simply connected and semisimple then $G$ is isomorphic to a product of simple simply connected groups $G_1$,\dots,$G_m$. Thus, $G$ contains a bad subgroup if and only if there exists $i$ such that $G_i$ contains a bad subgroup. Because of Schur's lemma, simple groups isomorphic to $\SL_n(\mathbb{C})$ do not contain bad subgroups and because of Table \ref{BdS} any other simply connected simple group contains a BdS subgroup and thus a bad subgroup by the preceding corollary. So the only CI-groups among simply connected semisimple groups are products of special linear groups. 

Thus, if $G$ is semisimple and a CI-group, the first sentence of this proof implies that the universal cover of $G$ has to be a product of special linear groups. Furthermore, one can construct a bad subgroup in any non-trivial quotient of a special linear group (see Lemma \ref{SLquo}). 

Therefore, $G$ is CI if and only if $DG$ is isomorphic to a product of special linear groups. 
\end{proof}

\begin{lem}\label{SLquo}
Let $G$ be a product of special linear groups and $C$ be a non-trivial central subgroup of $G$, then $G/C$ is not a CI-group.
\end{lem}
\begin{proof}  Let $n>1$, let $\xi$ be a primitive $n$-th root of the unity and $d$ dividing $n$. We define 
$$g_{n,d}=\lambda_{n,d}\begin{pmatrix}I_{\frac{n}{d}}&&&\\ &\xi^{\frac{n}{d}}I_{\frac{n}{d}}&&\\&&\ddots&\\&&&\xi^{\frac{n}{d}(d-1)}I_{\frac{n}{d}} \end{pmatrix}\text{ and }M_{n,d}:=\lambda_{n,d}\begin{pmatrix}&&&I_{\frac{n}{d}}\\ I_{\frac{n}{d}}&&&\\&\ddots&&\\&&I_{\frac{n}{d}}&\end{pmatrix}$$
where $\lambda_{n,d}$ is chosen so that $\det(g_{n,d})=\det(M_{n,d})=1$. It follows that $g_{n,d}$ and $M_{n,d}$ are in $\SL_n(\mathbb{C})$ and satisfy $[g_{n,d},M_{n,d}]=\xi^{\frac{n}{d}}$.

One sees that $M_{n,d}$ acts by conjugation on  the subgroup $D_{n,d}$ of $\SL_n(\mathbb{C})$ generated by unimodular matrices which are diagonal by blocks of size $n/d$.The group generated by $D_{n,d}$ and $M_{n,d}$ acts naturally on $\mathbb{C}^n$ and fixes no proper non-trivial subspace of $\mathbb{C}^n$. It follows that the group generated by  $D_{n,d}$ and $M_{n,d}$ is irreducible.

Let  $G$ be $\SL_{n_1}(\mathbb{C})\times\cdots\times \SL_{n_s}(\mathbb{C})$ and $C$ a non-trivial central subgroup of $G$. We take $c\in C$ such that $c\neq 1_G$ and write $c=(\xi_1^{\frac{n_1}{d_1}}I_{n_1},\dots,\xi_s^{\frac{n_s}{d_s}}I_{n_s})$ where $d_i$ divides $n_i$ and $\xi_i$ is a  $n_i$-th root of unity. We denote $\pi :G\to G/C$ the natural projection.

Let $g=(g_{n_1,d_1},\dots,g_{n_s,d_s})$, $M=(M_{n_1,d_1},\dots,M_{n_s,d_s})$ and $S$ be the group generated by $M$ and $D_{n_1,d_1}\times\cdots\times D_{n_s,d_s}$. Because the projection of $S$ for each factor of $G$ is irreducible, $S$ is itself irreducible and thus $\pi(S)$ is too. Furthermore $[g,M]=c$ by construction. Since $g$ commutes with $D_{n_1,d_1}\times\cdots\times D_{n_s,d_s}$, we deduce from this $Z_ {G/C}(\pi(S))$ contains $\pi(g)$. Since $g$ is not central in $G$, it follows that $G/C$ is not a CI-group.\end{proof}

\begin{rem}
In Section \ref{GBU-sec}, it is shown that $\X_r(G)^{red}\subset \X_r(G)^{sing}$ if $r\geqs 3$, or $r\geqs 2$ and the rank of the simple factors of the Lie algebra of $DG$ are at least 2.  Conversely, if $r=2$ there are semisimple Lie groups $G$ of arbitrarily large rank so $\X_r(G)$ contains smooth reducibles; \cite[Example 7.2]{FLR}.   These two facts together resolve the first part of \cite[Conjecture 3.34]{FL2}.  The second part of \cite[Conjecture 3.34]{FL2} states that $\X_r(G)^{red}=\X_r(G)^{sing}$ if and only if $DG$ is isomorphic to a product of special linear groups.  Given that we have shown in Section 2, that bad representations are singular whenever $r\geqs 3$, or $r\geqs 2$ and the rank of the simple factors of the Lie algebra of $DG$ are at least 2, this conjecture is equivalent to statement that the only CI groups are those whose derived subgroup is a product of special linear groups.  So the above theorem affirmatively resolves the second part of \cite[Conjecture 3.34]{FL2} too.
\end{rem}

As a result of the above theorem, Schur's Lemma (elements commuting with an irreducible subgroup are central) is true in only one simple $\C$-group: the special linear group. The main reason for CI-groups $G$ are interesting is that the irreducible locus of the $G$-character variety of a free group/surface group is a manifold (see \cite{FL2,Si4}). Next, we focus on the case when $G$ is simply connected.

The first lemma is fundamental to our discussion. It is true in greater generality than we state (see \cite[Chapter 4]{OnVi}). 

\begin{lem}
Let $G$ be a semisimple simply connected $\C$-group and $g$ a semisimple element in $G$. Then $Z_G(g)$ is connected. 
\end{lem}

\begin{proof}
A proof is given in \cite{Humphreys-conjugacy}, for example. 
\end{proof}

The next corollary is a direct consequence of this lemma.

\begin{cor}\label{bad-cor}
Let $G$ be a semisimple simply connected $\C$-group and $\rho:\F_r\to G$ a bad representation. Then $\rho(\F_r)$ is contained in a BdS subgroup. 
\end{cor}

\begin{proof}
Let $g$ be an element commuting with $\rho(\F_r)$. Since $\rho$ is irreducible, $g$ is semisimple (by Proposition \ref{commsemi}). From Lemma \ref{semisimplecentralizers}, it follows that $\mathfrak{z}_{\mathfrak{g}}(g)$ is either contained in a parabolic subalgebra or is a BdS subalgebra. Because of the preceding lemma, $Z_G(g)$ needs to be connected. As a result, if $\mathfrak{z}_{\mathfrak{g}}(g)$ were contained in a parabolic subalgebra then $Z_G(g)$ would be contained in a parabolic subgroup which would contradict the irreducibility of $\rho$. It follows that $\mathfrak{z}_{\mathfrak{g}}(g)$ is a BdS subalgebra and therefore $Z_G(g)$ is a BdS subgroup. 
\end{proof}

As a result, if we want to compute the bad locus of $G$-character varieties when $G$ is simply connected, it suffices to understand the irreducible characters (equivalence classes of irreducible representations) that factor through the inclusion of maximal BdS subgroups in $G$.

We now illustrate this principle with the lowest rank exceptional Lie group.  By definition (see for example \cite{B-S}, \cite{FH} or \cite{Rac}), $G_2$ is the automorphism group of a non-commutative, non-associative complex algebra $\mathbb{O}_\C:=\mathbb{O}\otimes_\R \C$ of complex dimension $8$ (the {\it bi-octonians}), where $\mathbb{O}$ is the usual octonians.

Since $G_2$ is simply connected, bad subgroups are contained in BdS subgroups by Corollary \ref{bad-cor}. From Table \ref{BdS}, we see that for $G_2$ there are only two types of BdS subgroups: type $A_2$ and $A_1\times A_1$. 

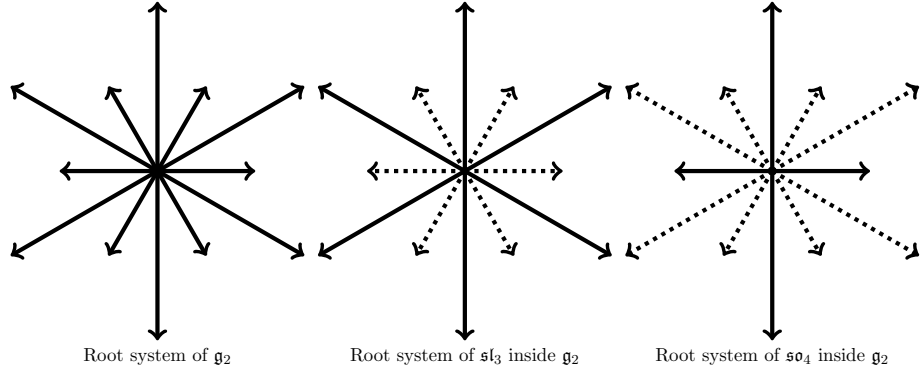
\begin{figure}[!ht]
\begin{tikzpicture}[scale=0.65]
\node[scale=0.65] at (0,-3.8) {Root system of $\mathfrak{g}_2$};
\draw[ultra thick,->,color=black] (0,0)--({2.0011*cos(0)},{2.0011*sin(0)});
\draw[ultra thick,->] (0,0)--({2.0011*cos(60)},{2.0011*sin(60)});
\draw[ultra thick,->] (0,0)--({2.0011*cos(120)},{2.0011*sin(120)});
\draw[ultra thick,->] (0,0)--({2.0011*cos(180)},{2.0011*sin(180)});
\draw[ultra thick,->] (0,0)--({2.0011*cos(240)},{2.0011*sin(240)});
\draw[ultra thick,->] (0,0)--({2.0011*cos(300)},{2.0011*sin(300)});

\draw[ultra thick,->](0,0)--({2.0011*1.73205081*cos(30)},{2.0011*1.73205081*sin(30)});
\draw[ultra thick,->](0,0)--({2.0011*1.73205081*cos(90)},{2.0011*1.73205081*sin(90)});
\draw[ultra thick,->,color=black](0,0)--({2.0011*1.73205081*cos(150)},{2.0011*1.73205081*sin(150)});
\draw[ultra thick,->](0,0)--({2.0011*1.73205081*cos(30)},{-2.0011*1.73205081*sin(30)});
\draw[ultra thick,->](0,0)--({2.0011*1.73205081*cos(90)},{-2.0011*1.73205081*sin(90)});
\draw[ultra thick,->](0,0)--({2.0011*1.73205081*cos(150)},{-2.0011*1.73205081*sin(150)});

\end{tikzpicture}
\begin{tikzpicture}[scale=0.65]
\node[scale=0.65] at (0,-3.8) {Root system of $\mathfrak{sl}_3$ inside $\mathfrak{g}_2$};
\draw[ultra thick,->,dotted] (0,0)--({2.0011*cos(0)},{2.0011*sin(0)});
\draw[ultra thick,->,dotted] (0,0)--({2.0011*cos(60)},{2.0011*sin(60)});
\draw[ultra thick,->,dotted] (0,0)--({2.0011*cos(120)},{2.0011*sin(120)});
\draw[ultra thick,->,dotted] (0,0)--({2.0011*cos(180)},{2.0011*sin(180)});
\draw[ultra thick,->,dotted] (0,0)--({2.0011*cos(240)},{2.0011*sin(240)});
\draw[ultra thick,->,dotted] (0,0)--({2.0011*cos(300)},{2.0011*sin(300)});

\draw[ultra thick,->](0,0)--({2.0011*1.73205081*cos(30)},{2.0011*1.73205081*sin(30)});
\draw[ultra thick,->](0,0)--({2.0011*1.73205081*cos(90)},{2.0011*1.73205081*sin(90)});
\draw[ultra thick,->,color=black](0,0)--({2.0011*1.73205081*cos(150)},{2.0011*1.73205081*sin(150)});
\draw[ultra thick,->](0,0)--({2.0011*1.73205081*cos(30)},{-2.0011*1.73205081*sin(30)});
\draw[ultra thick,->](0,0)--({2.0011*1.73205081*cos(90)},{-2.0011*1.73205081*sin(90)});
\draw[ultra thick,->](0,0)--({2.0011*1.73205081*cos(150)},{-2.0011*1.73205081*sin(150)});

\end{tikzpicture}
\begin{tikzpicture}[scale=0.65]
\node[scale=0.65] at (0,-3.8) {Root system of $\mathfrak{so}_4$ inside $\mathfrak{g}_2$};
\draw[ultra thick,->,color=black] (0,0)--({2.0011*cos(0)},{2.0011*sin(0)});
\draw[ultra thick,->,dotted] (0,0)--({2.0011*cos(60)},{2.0011*sin(60)});
\draw[ultra thick,->,dotted] (0,0)--({2.0011*cos(120)},{2.0011*sin(120)});
\draw[ultra thick,->] (0,0)--({2.0011*cos(180)},{2.0011*sin(180)});
\draw[ultra thick,->,dotted] (0,0)--({2.0011*cos(240)},{2.0011*sin(240)});
\draw[ultra thick,->,dotted] (0,0)--({2.0011*cos(300)},{2.0011*sin(300)});

\draw[ultra thick,->,dotted](0,0)--({2.0011*1.73205081*cos(30)},{2.0011*1.73205081*sin(30)});
\draw[ultra thick,->](0,0)--({2.0011*1.73205081*cos(90)},{2.0011*1.73205081*sin(90)});
\draw[ultra thick,->,color=black,dotted](0,0)--({2.0011*1.73205081*cos(150)},{2.0011*1.73205081*sin(150)});
\draw[ultra thick,->,dotted](0,0)--({2.0011*1.73205081*cos(30)},{-2.0011*1.73205081*sin(30)});
\draw[ultra thick,->](0,0)--({2.0011*1.73205081*cos(90)},{-2.0011*1.73205081*sin(90)});
\draw[ultra thick,->,dotted](0,0)--({2.0011*1.73205081*cos(150)},{-2.0011*1.73205081*sin(150)});

\end{tikzpicture}
\caption{BdS Subalgebras of $\mathfrak{g}_2$.}\label{g2root}
\end{figure}

In the second diagram in Figure \ref{g2root}, the sub-root system is of index $3$ while in the third diagram, the sub-root system is of index $2$. It follows (see Chapter 23 \S 2 in \cite{FH}) that the center of the BdS subgroup of type $A_2$ is of order $3$ and the center of the BdS subgroup of type $A_1\times A_1$ is of order $2$. Then, the two BdS subgroups are identified as  $\SL_3(\mathbb{C})$ and $\SO_4(\mathbb{C})$.

These two subgroups may be constructed using the minimal dimensional representation of $G_2$. 

The algebra $\mathbb{O}_\C=\mathbb{O}\otimes_\R \C$ contains a copy of $\mathbb{C}\otimes_{\mathbb{R}}\mathbb{C}$ and $\mathbb{H}\otimes_{\R}\C$ as subalgebras (where $\mathbb{H}$ is the quaternions). 
The subgroup $\SL_3(\mathbb{C})$ can be identified as the subgroup of $G_2$ that point-wise fixes the sub-algebra $\mathbb{C}\otimes_{\mathbb{R}}\mathbb{C}$, while $\SO_4(\mathbb{C})$ can be identified as the stabilizer of the sub-algebra 
$\mathbb{H}\otimes_{\R}\C$.

As a result bad representations in $G_2$ correspond to irreducible representations stabilizing non-degenerate sub-algebras of $\mathbb{O}_\C$. 

Lastly, note that the map from $\X_r(S)$ to $\X_r(G)$  induced by the inclusion of $S$ into $G$ has no reason to be injective in general. For instance, one may check that $\X_r(\SL_3(\mathbb{C}))^{irr}$ to $\X_r(G_2)^{irr}$ is $2$-to-$1$ onto its image. This follows from the fact that $\SL_3(\mathbb{C})$ is of index $2$ in its $G_2$-normalizer.   The corresponding map for $\SO_4(\mathbb{C})$ is more complicated. 

\newpage
\appendix
\section{Maximal Parabolic and BdS Subalgebras} \label{appa}

In this appendix, we compute the codimension of Lie subalgebras of simple Lie algebras referred to in the proof of Theorem \ref{codimbad}. In the first table, we consider the codimension of a Levi subalgebra $\mathfrak{l}$ in a maximal parabolic subalgebra of the corresponding simple Lie algebra.

\begin{table}[!ht]
$\begin{array}{|c|c|l|c|c|}
\hline
\mathfrak{g}&\dim_{\mathbb{C}}\mathfrak{g}&\multicolumn{1}{|c|}{[\mathfrak{l}_k,\mathfrak{l}_k]}&\codim_{\mathbb{C}}(\mathfrak{g},\mathfrak{l}_k)&\min_{k}\codim_{\mathbb{C}}(\mathfrak{g},\mathfrak{l}_k)\\
\hline\hline
A_r& r(r+2)&  A_{k-1}+A_{r-k},\ 1\leqs k\leqs r &2k(r+1-k)&2r \\ \hline
B_r& r(2r+1)& A_{k-1}+B_{r-k},\ 1\leqs k\leqs r &k(4r+1-3k) &2(2r-1) \\ \hline
C_r& r(2r+1)& A_{k-1}+C_{r-k},\ 1\leqs k\leqs r &k(4r+1-3k) &2(2r-1)\\ \hline
 
D_r& r(2r-1)& A_{k-1}+D_{r-k},\ 1\leqs k\leqs r-3 & k(4r-1-3k)&r(r-1),\text{ if }r=3,4\\
    &        & A_{r-3}+A_1+A_1,\ k=r-2&  r^2+3r-10  & 4(r-1),\text{ if }r> 4\\ 
    &        & A_{r-1},\ k=r-1,r & r^2-r & \\ \hline
    
 G_2& 14& A_1,\ k=1,2&10&10\\ \hline
 F_4& 52 &C_3,\ k=1 & 30 & 30 \\
    &     & A_1+A_2,\ k=2,3 & 40 & \\
    &      & B_3,\ k=4& 30 & \\ \hline
 E_6& 78 & D_5,\ k=1,5 & 32&32 \\
    &    & A_1+A_4,\ k=2,4&50 & \\
    &    & A_1+A_2+A_2,\ k=3&58 & \\
    &    & A_5,\ k=6 &42& \\ \hline
 E_7& 133 &D_6,\ k=1 &66 &54 \\
 &&A_1+A_5,\ k=2 &94 &\\
 &&A_1+A_2+A_3,\ k=3&106 &\\
 &&A_4+A_2,\ k=4&100 &\\
 &&D_5+A_1,\ k=5& 84&\\ 
 &&E_6,\ k=6&54 &\\
 &&A_6,\ k=7&84 & \\ \hline
 
 E_8& 248 &E_7,\ k=1      &114 & 114\\ 
    &     & A_1+E_6,\ k=2 &166 & \\ 
    &      &A_2+D_5,\ k=3 &194 & \\
    &      &A_3+A_4,\ k=4 &208 & \\
    &      &A_4+A_2+A_1,\ k=5&212 & \\
    &      &A_6+A_1,\ k=6 &196 & \\
    &      & D_7,\ k=7 &156 & \\
    &      &A_7, k=8& 184  &\\ \hline
 \end{array}$
 
\caption{Classification of Levi subalgebras in maximal parabolic subalgebras of simple Lie algebras.}\label{Levi}
\end{table}

To emphasize the ambient algebra, we write $\codim_{\mathbb{C}}(\mathfrak{g},\mathfrak{s})$ for the codimension of $\mathfrak{s}$ in $\mathfrak{g}$ in the above (and below) table.  

We recall, that once we choose a Cartan subalgebra $\mathfrak{h}$ and a set of simple roots $\{\alpha_1,\dots,\alpha_r\}$ to go with it, conjugacy classes of maximal parabolic subalgebras in simple Lie algebras are in one-to-one correspondence with the set of simple roots (see Chap. IV, 14.17 in \cite{Borel} for instance).  Now $\mathfrak{l}_k$ refers to a Levi subalgebra in the maximal parabolic subalgebra $\mathfrak{p}_k$ associated to the simple root $\alpha_k$ (this description of simple roots is as in Chapter 22 of \cite{FH}). 
 
Using this correspondence, the root system (and thus the isomorphism class)  of $[\mathfrak{l}_k,\mathfrak{l}_k]$ can easily be seen by removing the corresponding node on the Dynkin diagram. Finally, one uses the fact that for Levi subalgebras of maximal parabolic subalgebras, one has \[\mathfrak{l}_k=[\mathfrak{l}_k,\mathfrak{l}_k]\oplus\mathbb{C}.\]

Using the fact that $\codim_{\mathbb{C}}(\mathfrak{g},\mathfrak{l}_k)=2\codim_{\mathbb{C}}(\mathfrak{g},\mathfrak{p}_k)$, one also has the minimal codimension of parabolic subalgebras. 

In the second table, we compute the codimension of BbS subalgebras relevant to Theorem \ref{codimbad}. 

\begin{table}[!ht]
$\begin{array}{|c|c|l|c|c|}
\hline
\mathfrak{g}&\dim_{\mathbb{C}}\mathfrak{g}&\multicolumn{1}{|c|}{\mathfrak{s}}&\codim_{\mathbb{C}}(\mathfrak{g},\mathfrak{s})&\min_{\mathfrak{s}}\codim_{\mathbb{C}}(\mathfrak{g},\mathfrak{s})\\
\hline\hline
A_r& r(r+2)&  \multicolumn{1}{|c|}{\emptyset}&\emptyset &\emptyset \\ \hline
B_r& r(2r+1)& D_k+B_{r-k},\ 2\leqs k\leqs r& 2k(2r-2k-1)& 2r\\ \hline
C_r& r(2r+1)& C_k+C_{r-k} ,\ 1\leqs k\leqs r-1& 4k(r-k)&4(r-1),\ r\geqs 2\\ \hline
D_r& r(2r-1)& D_k+D_{r-k},\ 2\leqs k \leqs r-2& 4k(r-k)&8(r-2),\ r\geqs 3\\ \hline
G_2& 14& A_1+\widetilde{A_1},\ k=1& 8&6 \\ 
   &   & A_2,\ k=2                &6 & \\ \hline
F_4& 52 & A_1+C_3,\ k=1 &28 & 16\\
   &     &A_2+\widetilde{A_2},\ k=2 &36 & \\ 
   &     & A_3+\widetilde{A_1},\ k=3 &34 & \\
   &     & B_4,\ k=4 & 16& \\ \hline
E_6& 78 & A_5+A_1,\ k=2 &40 &40 \\ 
   &    & A_2+A_2+A_2,\ k=3 &54 & \\ \hline
E_7& 133 &D_6+A_1,\ k=1,6 &64 &64 \\
   &     &A_7,\ k=2 &90 & \\
   &     &A_5+A_2,\ k=3,5 &100 & \\
   &     &A_3+A_3+A_1,\ k=4 &70 & \\ \hline
E_8& 248 &D_8,\ k=1 &112 & 112\\ 
   &     &A_8,\ k=2 &162 & \\
   &      &A_1+A_7,\ k=3 &188 & \\ 
   &      &A_1+A_2+A_5,\ k=4 & 200&\\
   &      &A_4+A_4,\ k=5& 202&\\ 
   &      &D_5+A_3,\ k=6 & 182&\\ 
   &      &A_2+E_6,\ k=7 & 128& \\
   &      &A_1+E_7,\ k=8 & 168&\\ \hline
\end{array}$
\caption{Classification of maximal BdS subalgebras in simple Lie algebras.}\label{BdS}
\end{table}
We recall from \cite{Dyn} or \cite{Tit} that one can associate to any simple root of $\mathfrak{g}$ a BdS subalgebra of $\mathfrak{g}$. Furthermore, all maximal BdS subalgebras (if any) can be chosen among these subalgebras (however, not all such BdS subalgebras are maximal, see \cite{Tit}). 

In Table \ref{BdS} we use the same enumeration as in Table \ref{Levi}.  Also, the symbol $\widetilde{X}$ used in this table denotes a non-conjugate copy of the group $X$. The isomorphism class of the corresponding BdS subalgebra can also be read off the Dynkin diagram. In practice, one needs to add the minimal root of the root system to the Dynkin diagram and delete the $k$-th node to get the Dynkin diagram of the BdS subalgebra. One can compute its dimension from this.

\end{document}